\documentclass[11pt,a4paper]{article}

\usepackage{graphicx}        

\usepackage[a4paper, hmargin={2.7cm,2.7cm},vmargin={3.3cm,3.3cm}]{geometry}
\usepackage{amsmath, amssymb, amsfonts, amsbsy, latexsym, color,dsfont}
\usepackage{amscd,amsxtra}
\usepackage{amsthm}
\usepackage[normalem]{ulem}

\usepackage{listings}

\usepackage[T1]{fontenc}
\usepackage[utf8]{inputenc}
\usepackage[english]{babel}
\usepackage{cite}
\usepackage{paralist}
\usepackage{url}
\usepackage{twoopt}

\usepackage{pgfplotstable}
\usepackage{pgfplots}

\usepackage[bf,compact,pagestyles,small]{titlesec}
\titlespacing*{\section}{0pt}{14pt}{4pt}
\titlespacing*{\subsection}{0pt}{8pt}{3pt}

\usepackage{etoolbox}
\makeatletter
\patchcmd{\ttlh@hang}{\parindent\z@}{\parindent\z@\leavevmode}{}{}
\patchcmd{\ttlh@hang}{\noindent}{}{}{}
\makeatother

\usepackage{mathdots}

\usepackage{fancyhdr,lastpage}
\def\maketimestamp{\count255=\time
\divide\count255 by 60\relax
\edef\thetime{\the\count255:}%
\multiply\count255 by-60\relax
\advance\count255 by\time
\edef\thetime{\thetime\ifnum\count255<10 0\fi\the\count255}
\edef\thedate{\number\day-\ifcase\month\or Jan\or Feb\or Mar\or
             Apr\or May\or Jun\or Jul\or Aug\or Sep\or Oct\or
             Nov\or Dec\fi-\number\year}
\def\timstamp{\hbox to\hsize{\tt\hfil\thedate\hfil\thetime\hfil}}}
\maketimestamp
\fancypagestyle{plain}{%
\fancyhf{} 
\fancyfoot[C]{\small \thepage{} of \pageref{LastPage}}}

\numberwithin{equation}{section}  
\allowdisplaybreaks[4]

\newtheorem{theorem}{Theorem}[section]

\newtheorem{lemma}[theorem]{Lemma}
\newtheorem{proposition}[theorem]{Proposition}
\newtheorem{corollary}[theorem]{Corollary}
\theoremstyle{definition}
\newtheorem{example}[theorem]{Example}
\theoremstyle{remark}
\newtheorem{remark}[theorem]{Remark}

\DeclareMathOperator{\supp}{supp} %
\newcommandtwoopt{\mixedS}[2][\cG][\cH]{S_{{#1},{#2}}} 
\newcommand{\gaborA}[1][g]{\ensuremath{\set{M_{bm}T_{k} {#1}}_{m,k \in \Z}}} 
\newcommand{\gabor}[1][g]{\ensuremath{\set{M_{bm}T_{ak} {#1}}_{m,k \in \Z}}}
\newcommandtwoopt{\gaborG}[3][\Lambda][\Gamma]{\mathcal{G}(#3,#1,#2)} 

\newcommand{\charfct}[1]{\mathbf{1}_{#1}} 

\newcommand*{\numbersys}[1]{\ensuremath{\mathbb{#1}}}
\newcommand*{\C}{\numbersys{C}}
\newcommand*{\R}{\numbersys{R}}

\newcommand*{\Z}{\numbersys{Z}}

\newcommand*{\cH}{\mathcal{H}}

\newcommand*{\cG}{\mathcal{G}}

\newcommand{\itvoc}[2]{\ensuremath{\left({#1},{#2}\right]}} 
\newcommand{\itvoo}[2]{\ensuremath{\left({#1},{#2}\right)}} %
\newcommand{\itvcc}[2]{\ensuremath{\left[{#1},{#2}\right]}} %
\newcommand{\itvco}[2]{\ensuremath{\left[{#1},{#2}\right)}} %
\newcommand{\itvcos}[2]{\ensuremath{\lbrack{#1},{#2})}} %
\newcommand{\abs}[1]{\ensuremath{\left\lvert#1\right\rvert}}

\newcommand{\norm}[2][]{\ensuremath{\left\lVert#2\right\rVert_{#1}}}

\newcommand{\innerprod}[3][]{\ensuremath{\left\langle #2,#3\right\rangle_{\! #1}}}

\newcommand{\set}[1]{\ensuremath{\left\lbrace{#1}\right\rbrace}}

\newcommand{\setprop}[2]{\ensuremath{\left\lbrace{#1} : {#2}\right\rbrace}}

\newcommand{\kmax}{\ensuremath{k_{\mathrm{max}}}}
\usepackage{xspace}

\newcommand*\oline[1]{%
  \vbox{%
    \hrule height 0.5pt
    \kern0.25ex
    \hbox{%
      \kern-0.1em
      \ifmmode#1\else\ensuremath{#1}\fi
      \kern-0.1em
    }
  }
}

\usepackage{xcolor}

\usepackage{hyperref}
\hypersetup{
    colorlinks,
    linkcolor={red!50!black},
    citecolor={blue!50!black},
    urlcolor={blue!80!black},
 pdfview={FitH},
 pdfstartview={FitH},
 pdfauthor = {Jakob Lemvig, Kamilla Haahr Nielsen},
 pdftitle = {Gabor windows supported on [-1,1] and construction of compactly supported dual
  windows with optimal frequency localization},
 pdfkeywords = {dual frame, dual window, frame, Gabor system, optimal
   smoothness, redundancy},
 pdfcreator = {LaTeX with hyperref package},
 pdfproducer = PDFlatex}

\makeatletter
\def\blfootnote{\xdef\@thefnmark{}\@footnotetext} 
\def\subjclass{\xdef\@thefnmark{}\@footnotetext}
\long\def\symbolfootnote[#1]#2{\begingroup%
\def\thefootnote{\fnsymbol{footnote}}\footnote[#1]{#2}\endgroup} 
\if@titlepage
  \renewenvironment{abstract}{%
      \titlepage
      \null\vfil
      \@beginparpenalty\@lowpenalty
      \begin{center}%
        \bfseries \abstractname
        \@endparpenalty\@M
      \end{center}}%
     {\par\vfil\null\endtitlepage}
\else
  \renewenvironment{abstract}{%
      \if@twocolumn
        \section*{\abstractname}%
      \else
        \small
        \list{}{%
          \settowidth{\labelwidth}{\textbf{\abstractname:}}
          \setlength{\leftmargin}{50pt}
          \setlength{\rightmargin}{50pt}
          \setlength{\itemindent}{\labelwidth}
          \addtolength{\itemindent}{\labelsep}
        }
        \item[\textbf{\abstractname:}]

      \fi}
      {\if@twocolumn\else\endlist\fi}
\fi
\makeatother

\begin{document}

\title{Gabor windows supported on $[-1,1]$ and construction of compactly supported dual
  windows with optimal frequency localization}

\date{\today}

 \author{Jakob Lemvig\footnote{Technical University of Denmark, Department of Applied Mathematics and Computer Science, Matematiktorvet 303B, 2800 Kgs.\ Lyngby, Denmark, E-mail: \protect\url{jakle@dtu.dk}}\phantom{$\ast$},
Kamilla Haahr Nielsen\footnote{Technical University of Denmark, Department of Applied Mathematics
     and Computer Science, Matematiktorvet 303B, 2800 Kgs.\ Lyngby, Denmark, E-mail:
     \protect\url{kamni@dtu.dk}}} 


\maketitle

\thispagestyle{plain}
\begin{abstract} 
  We consider Gabor frames $\set{e^{2\pi i bm \cdot} g(\cdot-ak)}_{m,k \in \mathbb{Z}}$ with translation parameter $a=L/2$, modulation parameter $b \in (0,2/L)$ and a window function $g \in C^n(\mathbb{R})$ supported on $[x_0,x_0+L]$ and non-zero on $(x_0,x_0+L)$ for $L>0$ and $x_0\in \mathbb{R}$. The set of all dual windows $h \in L^2(\mathbb{R})$ with sufficiently small support is parametrized by $1$-periodic measurable functions $z$. Each dual window $h$ is given explicitly in terms of the function $z$ in such a way that desirable properties (e.g., symmetry, boundedness and smoothness) of $h$ are directly linked to $z$. We derive easily verifiable conditions on the function $z$ that guarantee, in fact, characterize, compactly supported dual windows $h$ with the same smoothness, i.e., $h \in C^n(\mathbb{R})$. The construction of dual windows is valid for all values of the smoothness index $n \in \mathbb{Z}_{\ge 0} \cup \{\infty\}$ and for all values of the modulation parameter $b<2/L$; since $a=L/2$, this allows for arbitrarily small redundancy $(ab)^{-1}>1$. We show that the smoothness of $h$ is optimal, i.e., if $g \notin C^{n+1}(\mathbb{R})$ then, in general, a dual window $h$ in $C^{n+1}(\mathbb{R})$ does not exist.
\end{abstract}

\section{Introduction}
\label{sec:introduction}

One of the central tasks in signal processing and time-frequency analysis is
to find convenient series expansions of functions in $L^2(\R)$. A
popular choice of such series expansions is by use of Gabor
frames, which are function systems of the form 
\[  
   \gabor = \set{e^{2\pi i bm \cdot} g(\cdot-ak)}_{m,k\in\Z}, 
\]
where $a,b>0$ and $g \in L^2(\R)$, and where $T_{\lambda} f = f(\cdot-\lambda)$ and $M_{\gamma} f = e^{2\pi i
  \gamma\cdot} f$, $\lambda,\gamma \in \R$, denote the translation and modulation operator on
$L^2(\R)$, respectively. Now, a \emph{Gabor frame} for $L^{2}(\R)$ is a
Gabor system $\gabor$ for which there exists constants $A,B>0$ such that
\begin{equation}
\label{eq:gabor-frame-def} A \norm{f}^2 \le \sum_{m,k\in \Z}
\abs{\innerprod{f}{M_{bm}T_{ak}g}}^2 \le B \norm{f}^2 \quad \text{for all} \quad f\in L^{2}(\R).
\end{equation}
If the upper bound holds, we say that $\gabor$ is a \emph{Bessel system} with
bound $B$.
In case $ab<1$ and $\gabor$
satisfies \eqref{eq:gabor-frame-def}, there exists infinitely many functions $h\in L^2(\R)$
 such that $\gabor[h]$ is a Bessel system and
\begin{equation} \label{eq:gabor-rep-formula}
 f = \sum_{m,k\in \Z}  \innerprod{f}{M_{bm}T_{ak}g} M_{bm}T_{ak}h  \ \ \text{for all } f\in L^2(\R)
\end{equation}
holds with unconditionally $L^2$-convergence. The function $g$ generating
the Gabor frame $\gabor$ is called the \emph{window}, while $h$ is called a \emph{dual
window}. For a Bessel system $\gabor$, the linear operator $S_g : L^2(\R) \to
L^2(\R)$ defined by 
\[
S_gf=\sum_{m,k\in \Z}  \innerprod{f}{M_{bm}T_{ak}g}
M_{bm}T_{ak}g
\]
 is called the \emph{frame operator}, and the
\emph{canonical dual window} is given by $\tilde{h} = S_g^{-1}g$. 
We will consider  windows $g$ supported on an interval of length $L$.

\subsection{Painless non-orthonormal expansions}
\label{sec:painl-non-orth}

The most classical method of constructing dual windows is by \emph{painless
non-orthonormal expansions} by Daubechies et al. \cite[Theorem
2]{DaubechiesPainless1986}. Assume $s(x)$ is nonnegative, bounded and
supported on an interval of length $L>0$, and that $s$ has constant
periodization $\sum_{n \in \Z} s(x+an)=1$ almost everywhere. Then   
defining $g=s^p$ and $h=s^{1-p}$, where $0<p<1$ classically is taken to
be $p=1/2$, generate dual frames $\gabor$ and $\gabor[h]$ for any $0<b
\le 1/L$ and $0<a\le L$. 

A variant of the painless construction without assuming a partition of
unity property (i.e., constant periodization) is a follows.
We will again consider $g \in L^\infty(\R)$ having compact support in
an interval of length $L$. 
If $a \le L$, and $b \le 1/L$, the frame operator
$S_g$ becomes a multiplication operator: 
\[ 
S_gf(x) = \frac{1}{b} \sum_{n \in \Z} \abs{g(x + an)}^2 \cdot f(x).
\]    
It follows that the Gabor system $\gabor$ is a frame with bound $A$ and $B$ if and only
if  $A \le \frac{1}{b} \sum_{n \in \Z} \abs{g(\cdot + an)}^2 \le B$,
in which case the canonical dual window $\tilde{h}:=S_g^{-1}g$ is compactly supported on
$\supp g$ and given by $\tilde{h}(x)= b g(x)/\sum_{n \in \Z} \abs{g(x
  + an)}^2$. In this case $\gabor$ and $\gabor[\tilde{h}]$ are
canonical dual Gabor frames with compact support in an interval of
length $L$.

\subsection{Our contribution}
\label{sec:our-contribution}
On the borderline $a=L$ and $0<b< 1/L$ of the painless expansions region $(a,b)\in \itvoc{0}{L} \times
\itvoc{0}{1/L}$, the discontinuous window $g=L^{-1/2}\chi_{\itvcc{0}{L}}$ generates
a tight Gabor frame. In this case, the Gabor system becomes a union of
Fourier series $\bigcup_{k \in \Z}\set{M_{bm}T_{ak}g}_{m \in\Z}$ with no support
overlap between the different Fourier systems $\set{M_{bm}T_{ak}g}_{m
  \in\Z}$ indexed by $k\in\Z$. This 
may lead to unwanted mismatch artifacts at the seam points
$a\Z$ when truncating the Gabor expansion~\eqref{eq:gabor-rep-formula}. 

To diminish such artifacts, we will use a $2$-overlap (or $2$-covering) condition, namely, 
$a=L/2$. This means that for almost any time $x\in \R$, \emph{two}
Fourier-like systems $\set{M_{bm}T_{ak}g}_{m \in\Z}$, $k \in \Z$,
represent the signal $f$ at time $x$. Phrased differently, for each $k \in \Z$, the Fourier-like
system $\set{M_{bm}T_{ak}g}_{m \in \Z}$ has an overlap of length
$L/2$ with $\set{M_{bm}T_{a(k-1)}g}_{m \in \Z}$ and
with $\set{M_{bm}T_{a(k+1)}g}_{m \in \Z}$. 

Under the $2$-overlap condition $a=L/2$, the painless construction is
only applicable for redundancies $(ab)^{-1}\ge
2$ since $b \le 1/L$. We are interested in small
redundancies $(ab)^{-1}<2$, that is, large modulations $b > 1/L$; note
that, as is standard in Gabor analysis, in fact, necessary once
continuity of the window $g$
is imposed, we always take $(ab)^{-1}>1$. Hence, we are interested in the
redundancy range $(ab)^{-1} \in \itvoo{1}{2}$. However, once outside
the region of painless non-orthonormal expansions, computing the canonical dual $\tilde{h}=S_g^{-1}g$ becomes much more
cumbersome since it requires inverting the frame
operator $S_g$, and one often resort to numerical
approaches~\cite{JanssenIterative2007,PerraudinDesigning2018,Moreno-PicotEfficient2018,KloosImplementation2016,BalazsDouble2006}. 

For $n \in \Z_{\ge 0}\cup\{\infty\}$, we consider the window class of $n$ times continuously differentiable
functions $g:\R\to \C$, i.e., $g \in C^n(\R)$, that are supported on
a closed interval $\itvcc{x_0}{x_0+L}$ ($x_0 \in \R$) of length
$L$ and nonzero on the open interval $\itvoo{x_0}{x_0+L}$. For any function $g$ in this window class, the Gabor system $\set{M_{bm}T_{Lk/2}g}_{m \in \Z}$ is a frame for $L^2(\R)$
 if and only if $b \in \itvoo{0}{2/L}$
 \cite[Corollary~2.8]{ChristensenGabor2010}. 
The objective of our work is to
characterize and construct compactly
supported (alternate) dual windows $h$ of $g$ with good, even optimal,
frequency localization in a setting beyond the painless expansions region
with arbitrarily small redundancy, but without having to invert the frame operator.
 The
main features of our approach can be summarized as follows: 
\begin{enumerate}[(I)]
\item It uses 2-overlap ($a=L/2$) and works outside the region of
  painless non-orthonormal expansions and with arbitrarily small redundancy $(ab)^{-1}>1$ of the Gabor frames.
\item It provides a natural parametrization of all dual windows
  (via an explicit formula) with sufficiently small support, given in terms of a measurable function
  $z$ defined on a compact interval. 
\item  It provides optimal smoothness of the dual window $h$, e.g.,
  dual windows $h$ with the same smoothness as the
  original window, i.e., $g,h \in C^n(\R)$.  
\item  It yields support size of $h$ only depended on the modulation
  parameter $b \in \itvoo{0}{2/L}$, not on properties of $g$ (or $h$,
  e.g., smoothness). 
\end{enumerate}

Note that since we work beyond the painless expansions region (that is, $b>1/L$),
the canonical dual may not even have compact support. By a result of
B\"olcskei~\cite{BoelcskeiNecessary1999}, assuming rational redundancy
$(ab)^{-1}$, the canonical dual window has compact support if and only
if  the Zibulski-Zeevi representation of the frame
operator $S_g$ is unimodular in the frequency variable. 

\subsection{Results in the literature}
\label{sec:results-literature}

By dilation and translation of the Gabor system, 
we may without loss of generality take the translation parameter $a=1$, the
modulation
parameter $b \in \itvoo{0}{1}$, and $\supp g =\itvcc{-1}{1}$. 
Christensen, Kim and Kim \cite{ChristensenGabor2010} characterize the
frame property of $\gaborA$ for $g \in C^0(\R)$ with $\supp g =
\itvcc{-1}{1}$ and finitely many zeros in $\itvoo{-1}{1}$. In
particular, they inductively construct a continuous and compactly
supported dual window once such a $g$ generates a Gabor
frame $\gaborA$. While the focus in \cite{ChristensenGabor2010} is on
existence questions, we aim for explicit constructions for 
dual windows with symmetry and higher order smoothness, albeit for a smaller
class of window functions as we do not allow $g$ to have zeros inside
the support.  

Gabor systems with continuous windows supported on an interval of length $2$ have
 been considered in several recent
 papers~\cite{ChristensenGabor2015,Atindehouframe2018b} 
 typically the dual windows constructed in these works are not continuous.

There exist a considerable amount of work in Gabor analysis on
explicit constructions of alternate dual windows, see, e.g., \cite{ChristensenPairs2006,ChristensenPairs2006a,ChristensenGabor2010,ChristensenGabor2012,Christensendual2014,LaugesenGabor2009,KimGabor2015}. These constructions have the desirable feature that the dual window
shares or inherits many of the properties of the window, e.g.,
smoothness and compact support.
However, the methods from all the above cited works are restricted to the
painless expansions region, e.g., redundancies $(ab)^{-1} \ge 2$ or even
$(ab)^{-1} \ge 3$. 

The construction in \cite{LaugesenGabor2009} of  $C^n$-smooth dual windows 
supported on $\itvcc{-1}{1}$ is not guaranteed to work, however, when
it does both windows are spline polynomials. Our method always works,
but if $g$ is a piecewise polynomial, the dual window will in general
be a piecewise rational function of polynomials. The construction of
Laugesen~\cite{LaugesenGabor2009} is generally limited to the painless
region, but by using a trick of
non-linear dilation by a soft-thresholding type function,
 Laugesen is able to handle smaller redundancies, i.e.,
 $1<(ab)^{-1}<2$. However, this has
the effect of making the windows constant of most of their support,
which may not be desirable. 

In short, our work can be seen as a a continuation of
\cite{ChristensenGabor2010,ChristensenGabor2012}, but with
the objective of \cite{LaugesenGabor2009} to construct smooth dual
pairs of Gabor windows.

\subsection{Outline}
\label{sec:outline}

In Section~\ref{sec:construction} we introduce the family of dual
windows. Section~\ref{sec:prop-dual-wind} is the main contribution with
a detailed analysis of properties (smoothness, symmetry, etc.) of the dual windows. In
Section~\ref{sec:examples} we present examples of the construction.

\section{The construction of the dual windows} 
\label{sec:construction}

For each $n \in \Z_{\ge 0} \cup \set{\infty}$ we define the window classes:
\begin{equation}
  \label{eq:def-Vn_pos}
   V_+^n(\R)=\setprop{f \in C^n(\R)}{\supp{f} =\itvcc{-1}{1} \text{
       and } \abs{f(x)}>0 \text{ for all } x \in \itvoo{-1}{1}}
\end{equation}
Observe that the window classes are nested
$V_+^{n}(\R) \subset V_+^{n-1}(\R)$ for $n \in \Z_{>0}$, and that even for
$g$ in the largest of these window classes $V_+^0(\R)$, it is known
\cite[Corollary 2.8]{ChristensenGabor2010} that the Gabor system $\gaborA$ is a
frame for $L^2(\R)$ for any $b \in \itvoo{0}{1}$.



We now introduce compactly supported functions $h_z$ that will serve
as dual windows of $g \in V^0_+(\R)$. Assume
$0 < b<1$. Let $\kmax \in \Z_{\ge 0}$ be the largest integer
strictly smaller than $b/(1-b)$, that is,
\[ 
   \kmax = \max \setprop{k \in \Z_{\ge 0}}{k<\frac{b}{1-b}}.
\]
Note that $\kmax \ge 1$ when $1/2 < b <1$. For any
$k\in \set{0, 1,\dots, \kmax}$ we have $k(1/b-1)<1$. We
define, for $k\in \set{1,2,\dots, \kmax}$,
\[
[k]=\{1,2,\dots,k\},
\]
and set $[0]=\emptyset$.

For $g \in V_+^0(\R)$, define $\psi: \R \to \C$ by
\begin{equation}
  \label{eq:def-psi}
    \psi(x) = \frac{1}{\sum_{n \in \Z} g(x+n)} \qquad \text{for } x \in \R.
\end{equation}
By the assumptions on $g$, the $1$-periodic function $\psi$ is
well-defined, continuous, and satisfies
 \begin{equation}
   c \le \psi(x) \le C \qquad \text{for all } x \in \R
\label{eq:psi-bounds}
 \end{equation}
for some positive, finite constants $c,C>0$. We will often consider 
$\psi$ as a function on $\itvcc{0}{1}$ with $\psi(0)=\psi(1)=1/g(0)$ given by $\psi(x) =
1/(g(x)+g(x-1))$ for $x \in \itvcc{0}{1}$.

Let $z : \itvcc{0}{1}\to \C$ be a measurable function.  For each
$k \in \set{0,1,\dots, \kmax}$ we define the following auxiliary
functions:
\begin{align}
  \eta_k(x-k) = (-1)^{k} \prod_{j \in [k]}
  \frac{g(x+1+j(1/b-1))}{g(x+j(1/b-1))} \left[ 
-g(x+1) z(x+1) + b\psi(x+1) \right] 
\intertext{for $x \in \itvcc{-1}{-k(1/b-1)}$ and}
  \gamma_{k}(x+k) = (-1)^{k} \prod_{j \in [k]}
  \frac{g(x-1-j(1/b-1))}{g(x-j(1/b-1))} \left[ 
g(x-1) z(x) + b\psi(x) \right] \label{def_gamma}
\end{align}
for $x \in \itvcc{k(1/b-1)}{1}$. We
finally define $h_z:\R\to \C$ by
\begin{align}
\overline{h_z(x)} =  \sum_{k=0}^{\kmax} \eta_{k}(x)\, \charfct{\itvcc{-k-1}{-k/b}}(x) 
             +   \sum_{k=0}^{\kmax} \gamma_k(x)\,
  \charfct{\itvcc{k/b}{k+1}}(x) \quad \text{for } x \in \R\setminus\{0\}, 
  \label{eq:def-h}
 \end{align}
and $\overline{h_z(0)}=b\psi(0)$. More explicitly, $h_z$ is given as:
\begin{equation*}
\overline{h_z(x)} =
\begin{cases} 
                \eta_k(x) &   x \in \itvcc{-k-1}{-k/b}, k = 1,\dots, \kmax,  \\
                -g(x+1)z(x+1)+ b\psi(x+1) & x \in \itvco{-1}{0}, \\
                g(x-1)z(x)+ b\psi(x) & x \in \itvcc{0}{1}, \\
                \gamma_k(x) &   x \in \itvcc{k/b}{k+1}, k = 1,\dots,
                \kmax,  \\
                  0 & \text{otherwise}.
   \end{cases} 
 \end{equation*}

 We remark that the function $\gamma_k(\cdot+k)$ is indeed well-defined on
 $\itvcc{k(1/b-1)}{1}$ since the product
 $\prod_{j \in [k]} \frac{1}{g(\cdot-j(1/b-1))}$ is well-defined on
 $\itvoo{-1+k(1/b-1)}{1/b}$ and
 $\itvcc{k(1/b-1)}{1} \subset \itvoo{-1+k(1/b-1)}{1/b}$. Note also
 that the product $\prod_{j \in [k]} g(\cdot-1-j(1/b-1))$ has support
 on $\itvcc{k(1/b-1)}{1+1/b}$. A similar consideration shows that
 $\eta_k$ is well-defined on its domain.
 
The function $h_z$ defined in \eqref{eq:def-h} has compact support in $\itvcc{-\kmax-1}{\kmax+1}$. More
precisely,
\begin{equation}
  \supp{h_z} \subset \bigcup_{k=1}^{\kmax} \itvcc{-k-1}{-k/b} \cup
  \itvcc{-1}{1} \bigcup_{k=1}^{\kmax} \itvcc{k/b}{k+1}.
\label{eq:support-of-h}
\end{equation}
The function $h_z$ is piecewise defined with
$\cup_{k=0}^{\kmax} \set{\pm k/b,\pm(k+1)}$ being the \emph{seam
  points} of $h_z$. The seam point $x=0$ is special as it is the only
seam point, where two nonzero functions in the definition of $h_z$
meet. In all other seam points, i.e., $x\neq 0$, the function $h_z$ is zero on
one side of the seam points.

\begin{remark}
\label{rem:difference-betwe-N-kmax} 
For later use, we remark that Christensen, Kim, and Kim~\cite{ChristensenGabor2010} consider dual windows
$h \in C^0(\R)$ of $g \in V_+^0(\R)$ with compact support
$\supp h \subset \itvcc{-N}{N}$, where
$N := \max \setprop{n \in \Z_{>0} }{n \le b/(1/b)}+1$ and $b\ge 1/2$. When
$b/(1-b)\notin \Z_{>0}$, then $N=\kmax+1$, and the support
$\supp h \subset \itvcc{-N}{N}$ corresponds to the support of $h_z$ in
\eqref{eq:support-of-h}. On the other hand, if $b/(1-b)\in \Z_{>0}$, there
is a mismatch in the support relation as $N=\kmax+2$. However, in case
$b/(1-b) \in \Z_{>0}$ we have $(N-1)/b=N$, and the interval
$\itvcc{(N-1)/b}{N}$ collapses to a point $\set{N}$. As a consequence,
we can, in any case, use $N:=\kmax+1$ when applying results from
\cite{ChristensenGabor2010}.
\end{remark}

\section{Properties of the dual windows} 
\label{sec:prop-dual-wind}

We first prove that $h_z$ indeed is a dual window of $g$ as soon as
$\gaborA[h_z]$ is a Bessel system in Section~\ref{sec:duality}.  In
Section~\ref{sec:properties-h_z} we show that the chosen
parametrization $z \mapsto h_z$ has several desirable properties. In
Section~\ref{sec:high-order-smoothn} we show how to construct smooth
dual windows $h \in C^n(\R)$ for any $g \in V^n_+(\R)$ and
$n \in \Z_{\ge 0} \cup \{\infty\}$. In Section~\ref{sec:optimality} we
discuss optimal smoothness of dual windows and show that the results
in Section~\ref{sec:high-order-smoothn}, in general, are optimal.
In Section~\ref{sec:duals-with-small} we minimize the support length
of the dual window in the sense of \cite{ChristensenGabor2012} while
preserving the optimal smoothness.


However, we first introduce some notation used below.  
We use $f(x_0^+)$ to denote the one-sided limit from
the right $\lim_{x\searrow x_0} f(x)$ and similarly $f(x_0^-)$ to denote
the one-sided limit from the left $\lim_{x\nearrow x_0} f(x)$. For $n
\in \Z_{\ge 0}$, we let $D^n[f]$ and $f^{(n)}$ denote the $n$th derivative of a
function $f : \R \to \C$, with the convention $f(x)=D^0[f](x)=f^{(0)}(x)$. 
We will repeatedly use the (general) Leibniz rule for differentiation
of products:
\begin{equation}
 (fg)^{(n)}(x)=\sum _{\ell=0}^{n}{n \choose
   \ell}f^{(n-\ell)}(x)g^{(\ell)}(x),
\label{eq:leibniz}
\end{equation}
where ${n \choose \ell}=\frac{n!}{\ell!(n-\ell)!}$ is the binomial coefficient.

A function $f:\itvcc{c}{d}\to \C$ is said to piecewise $C^{n}$ if there exists a
finite subdivision $\set{x_0,\dots,x_k}$ of $\itvcc{c}{d}$, $x_0=c$,
$x_n=d$ such that $f$ is $C^n$ on $\itvcc{x_{i-1}}{x_{i}}$, the derivatives at $x_{i-1}$ understood
as right-handed and the derivatives at $x_{i}$ understood
as left-handed, for every $i \in \set{1,\dots,k}$. Hence, if
$f$ is piecewise $C^n$, the one-sided limits of $f^{(m)}$ exists
everywhere for all $m=0,\dots,n$, but the left and right limits may
differ in a finite number of points. The definition of piecewise $C^n$
functions can be extended to functions on $\R$, we will, however,
only need it for compactly supported functions, where the modification
is obvious.

\subsection{Duality}
\label{sec:duality}

The following duality condition for two Gabor systems by Ron and
Shen~\cite{RonFrames1995,RonWeyl1997} is central to our work; we use the
formulation due to Janssen\cite{Janssenduality1998}.
\begin{theorem}[\cite{RonFrames1995}]
Let $b>0$ and $g,h \in L^2(\R)$. Suppose $\gaborA$ and $\gaborA[h]$
are a Bessel sequences. Then $\gaborA$ and $\gaborA[h]$ are dual frames for
$L^2(\R)$, if and only if, for all $k \in \Z$,
  \begin{equation}
    \label{eq:char-eqns}
    \sum_{n\in\Z} g(x+k/b+n) \overline{h(x+n)} = \delta_{0,k}b \quad \text{for a.e. } x \in \R.
  \end{equation}
\end{theorem}
Since the infinite sums in \eqref{eq:char-eqns} are $1$-periodic, it
suffices to verify \eqref{eq:char-eqns} on any interval of length one.    
Furthermore, for $\supp g \subset \itvcc{-1}{1}$, the duality conditions
\eqref{eq:char-eqns} become, for $k \neq 0$,
  \begin{align}
    \label{eq:char-eqns-k-not-0}
    g(x-k/b) \overline{h(x)} + g(x-k/b-1) \overline{h(x-1)} &= 0 \quad \text{for a.e. } x \in \itvcc{k/b}{k/b+1}
\intertext{and, for $k=0$,} 
    \label{eq:char-eqns-k-0}
    g(x) \overline{h(x)} + g(x-1) \overline{h(x-1)} &= b \quad \text{for a.e. } x \in \itvcc{0}{1}.
 \end{align}

The following theorem shows that $h_z$ is a convenient representation of
dual windows of $g$ and justifies our interest in $h_z$.

\begin{theorem}
  Let $b \in \itvoo{0}{1}$, let 
  $g \in V^0_+(\R)$, and let $z:\itvcc{0}{1}\to \C$ be a measurable
  function. Then $g$ and $h_z$ satisfy the characterizing equations \eqref{eq:char-eqns}.
Hence, if $\gaborA[h_z]$ is a Bessel sequence, e.g., if $h_z \in
L^\infty(\R)$, then $\gaborA$ and $\gaborA[h_z]$ are dual frames for
$L^2(\R)$. 
\end{theorem}
\begin{proof}
  First, we check that \eqref{eq:char-eqns} holds for $k=0$. Since 
  \begin{equation*}
    g(x)\overline{h_z(x)} =
	\begin{cases} 
		g(x)\left[-g(x+1)z(x+1)+ b \psi(x+1)\right] & a.e.\ x \in \itvco{-1}{0} \\
		g(x)\left[g(x-1)z(x)+ b \psi(x)\right] & a.e.\ x \in
                \itvco{0}{1} \\
        0 & \text{ otherwise,}
    \end{cases}
  \end{equation*}
 it follows that, for $a.e.\ x \in \itvco{0}{1}$,
  \begin{align*}
    \sum_{n\in\Z} g(x+n) \overline{h_z(x+n)} &= -g(x-1)g(x)z(x)+ g(x-1) b \psi(x)\\[-1.1em]
                & \qquad + g(x)g(x-1)z(x)+ g(x) b \psi(x)\\
                &= b \psi(x) \left(g(x-1)+ g(x)\right) = b,
  \end{align*}
  where the final equality follows from the definition of $\psi(x)$. This shows that \eqref{eq:char-eqns} holds for $k=0$. 

For $\abs{k} \ge \kmax +1$, the functions $g(\cdot-k/b)$ and $h_z$ have
disjoint support. This follows from \eqref{eq:support-of-h} and the
fact that $\kmax + 1 \le (\kmax+1)/b$. Consequently,
equation~\eqref{eq:char-eqns} holds for $\abs{k} \ge \kmax +1$.

To show that \eqref{eq:char-eqns} holds $k \in \set{1,\dots,\kmax}$,
we will verify \eqref{eq:char-eqns} on $\itvco{k}{k+1}$. We first
compute 
  \begin{equation*}
    g(x-k/b)\overline{h_z(x)} =
	\begin{cases} 
		g(x-k/b)\gamma_{k-1}(x) & a.e.\ x \in \itvco{k/b-1}{k}, \\
		g(x-k/b)\gamma_{k}(x) & a.e.\ x \in \itvco{k/b}{k+1}, \\
		0 & \text{otherwise.}
    \end{cases}
  \end{equation*}      
  For $a.e.\ x \in \itvco{k}{k/b}$, it is trivial that
  \begin{equation*}
    \sum_{n\in\Z}g(x-k/b+n)\overline{h_z(x+n)} = 0.
  \end{equation*}  
  On the other hand, for $a.e.\ x \in \itvco{k/b}{k+1}$,
  \begin{equation}\label{eqn:CharEqn-pos-k}
    \sum_{n\in\Z}g(x+\tfrac{k}{b}+n)\overline{h_z(x+n)} = g(x-\tfrac{k}{b})\gamma_{k}(x) + g(x-\tfrac{k}{b}-1)\gamma_{k-1}(x-1)
  \end{equation}
We focus on the first term of the right hand side, which, by definition,
is given as:
  \begin{equation*}
  g(x-\tfrac{k}{b})\gamma_{k}(x) = g(x-\tfrac{k}{b})(-1)^{k} \prod_{j \in [k]}
  \frac{g(x-k-1-j(\tfrac{1}{b}-1))}{g(x-k-j(\tfrac{1}{b}-1))} \left[ 
  g(x-k-1) z(x-k) + b \psi(x-k)\right].
  \end{equation*}
Since
  \begin{equation*}
  \prod_{j \in [k]} \frac{g(x-k-1-j(\tfrac{1}{b}-1))}{g(x-k-j(\tfrac{1}{b}-1))}
  =\frac{g(x-\tfrac{k}{b}-1)}{g(x-\tfrac{k}{b})}\prod_{j \in [k-1]} \frac{g(x-k-1-j(\tfrac{1}{b}-1))}{g(x-k-j(\tfrac{1}{b}-1))},
  \end{equation*}
we can rewrite $g(x-\tfrac{k}{b})\gamma_{k}(x)$  as follows: 
  \begin{align*}
  g(x-\tfrac{k}{b})\gamma_{k}(x) &= g(x-\tfrac{k}{b}-1)(-1)^{k} \prod_{j \in [k-1]}
  \frac{g(x-k-1-j(\tfrac{1}{b}-1))}{g(x-k-j(\tfrac{1}{b}-1))} \\ & \phantom{=
                                                 g(x-\tfrac{k}{b}-1)(-1)^{k}
                                                 \prod_{j \in
                                                 [k-1]}} \cdot \left[ 
  g(x-k-1) z(x-k) + b \psi(x-k)\right]\\
  &= -g(x-\tfrac{k}{b}-1)\gamma_{k-1}(x-1).
  \end{align*}
By inserting this back into \eqref{eqn:CharEqn-pos-k}, we obtain, for $a.e.\ x \in \itvco{k/b}{k+1}$,
\[
\sum_{n\in\Z}g(x+\tfrac{k}{b}+n)\overline{h_z(x+n)} =
-g(x-\tfrac{k}{b}-1)\gamma_{k-1}(x-1) +
g(x-\tfrac{k}{b}-1)\gamma_{k-1}(x-1)=0,
\] 
which verifies (\ref{eq:char-eqns}) for $k \in \set{1, \dots, \kmax}$.

The calculations for $k \in \set{-\kmax, \dots, -1}$ are similar to the
above calculations, hence we leave this case for the reader. 
\end{proof}

The next result,
Lemma~\ref{lem:para-all-duals}, shows that not only is $h_z$ a
convenient expression of dual windows of $g \in V_+^0(\R)$, it is in
fact a parametrization by $1$-periodic measurable functions $z$ of
\emph{all} dual windows with sufficiently small support. The structure
of the proof of Lemma~\ref{lem:para-all-duals} is somewhat similar to the
proof of Lemma~3.3 in \cite{ChristensenGabor2010}. As the two results
are quite different, we give the proof of
Lemma~\ref{lem:para-all-duals}.

\begin{lemma}[Parametrization of all compactly supported dual windows]
\label{lem:para-all-duals}
  Let $b \in \itvoo{0}{1}$ and $g \in V_+^0(\R)$. Suppose $h \in L^2(\R)$ has compact support in
  $\itvcc{-\kmax-1}{\kmax+1}$. If $\gaborA$ and $\gaborA[h]$ are dual
  frames, then $h=h_z$ for some measurable function $z:\itvcc{0}{1} \to
  \C$. 
\end{lemma}
\begin{proof}
  Let $h \in L^2(\R)$ be a dual window of $g$ with
 $\supp h \subset \itvcc{-\kmax-1}{\kmax+1}$. 
 Lemma 3.2 in \cite{ChristensenGabor2010} says, see Remark~\ref{rem:difference-betwe-N-kmax},  that 
\[ 
  \supp{h} \subset \bigcup_{k=1}^{\kmax} \itvcc{-k-1}{-k/b} \cup
  \itvcc{-1}{1} \bigcup_{k=1}^{\kmax} \itvcc{k/b}{k+1}.
\]
  Define a measurable function $z$ on $\itvcc{0}{1}$ by:
 \[ 
  z(x) =  \frac{\overline{h(x)}-b\psi(x)}{g(x-1)} \qquad a.e.\ x \in \itvoo{0}{1}.
\] 
 This definition gives immediately that $h(x)=h_z(x)$ for $a.e.\ x \in \itvcc{0}{1}$.

By \eqref{eq:char-eqns} with $k=0$, we have for $a.e.~x \in \itvoo{-1}{0}$
\begin{align*}
  \overline{h(x)}&=\frac{-g(x+1)\overline{h(x+1)}+b}{g(x)} 
\intertext{We continue, using that $h(x)=h_z(x)$ for a.e.\@ $x \in
                   \itvcc{0}{1}$,}
  \overline{h(x)}&=\frac{-g(x+1)\bigl[-g(x)z(x+1) + b \psi(x+1)\bigr]+b}{g(x)} \\
 & = -g(x+1)z(x+1) + b \psi(x+1) = \overline{h_z(x)} \qquad \text{for
   } x \in \itvoo{-1}{1}. 
\end{align*}
Thus, also $h(x)=h_z(x)$ for a.e.~$x \in \itvcc{-1}{1}$. 

We will complete the proof by induction, showing that 
\[ 
h(x)=h_z(x) \quad \text{for a.e. } x \in \bigcup_{k=1}^{k_0} \itvcc{-k-1}{-k/b} \cup
  \itvcc{-1}{1} \bigcup_{k=1}^{k_0} \itvcc{k/b}{k+1}.
\]
holds for all integers $k_0$ in $1\le k_0 \le \kmax$. The base case $k_0=0$ was
verified above so we only have to show the induction step $k_0-1 \to
k_0$. 

We first consider $x>0$. By induction hypothesis, we have
$\overline{h(x)}=\overline{h_z(x)}=\gamma_{k_0-1}(x)$ for $a.e.\ x \in
\itvcc{(k_0-1)/b}{k_0}$. 
We aim to prove $\overline{h(x)}=\gamma_{k_0}(x)$ for $a.e.\ x \in
\itvcc{k_0/b}{k_0+1}$, or rather $\overline{h(x+1)}=\gamma_{k_0}(x+1)$
for $a.e.\ x \in
\itvcc{k_0/b-1}{k_0}$. Since $\itvcc{k_0/b-1}{k_0} \subset \itvcc{k_0/b-1}{k_0/b}$, it
follows by \eqref{eq:char-eqns} for $k=-k_0$:
\begin{align*}
  \overline{h(x+1)}&=-\frac{g(x-k_0/b)}{g(x-k_0/b+1)}\cdot \overline{h(x)}
 \intertext{for $a.e.\ x \in \itvcc{k_0/b-1}{k_0}$. Since also $\itvcc{k_0/b-1}{k_0} \subset \itvcc{(k_0-1)/b}{k_0}$, using the induction
                     hypothesis yields further:}
  \overline{h(x+1)}&=
                     -\frac{g(x-k_0-k_0(1/b-1))}{g(x-k_0+1-k_0(1/b-1))} \cdot
                     \gamma_{k_0-1}(x) = \gamma_{k_0}(x+1)
                     = \overline{h_z(x+1)}
\end{align*}
for $a.e.\ x \in \itvcc{k_0/b-1}{k_0}$. The argument for $x<0$ is
similar, hence we omit it.   
\end{proof}

Obviously, there are choices of an unbounded function $z$ that
leads to $h_z$ not generating a Gabor Bessel sequence in $L^2(\R)$,
e.g., if $h_z
\notin L^2(\R)$, in which case $\gaborA$ and
$\gaborA[h_z]$ are not dual
  frames. This is not contradicting Lemma~\ref{lem:para-all-duals}. 

Lemma~\ref{lem:para-all-duals} should be compared to the well-known
parametrization of all dual windows $h$ by functions $\varphi$ generating Bessel
Gabor systems, see, e.g.,
\cite[Proposition~12.3.6]{Christensenintroduction2016}; the
parametrization formula~\eqref{eq:all-duals-li} is due to Li~\cite{Ligeneral1995} and reads:
\begin{equation}
   h = S_g^{-1} g + \varphi - \sum_{m,n \in \Z}
   \innerprod{S^{-1}g}{M_{bm}T_{ak}g}M_{bm}T_{ak}\varphi. \label{eq:all-duals-li}
 \end{equation}
The parametrization by Bessel generators $\varphi$ is not injective nor explicit as one needs to compute both $S_g^{-1}
g$ and an infinite series. On the other hand, the
mapping $z \mapsto h_z$ ($z$ measurable on $\itvcc{0}{1}$) is 
injective and its range contains all dual windows with compact support
on $\itvcc{-\kmax-1}{\kmax+1}$. Moreover, by Lemma~\ref{lem:boundedness} below, the mapping $L^\infty({\itvcc{0}{1}}) \ni z \mapsto h_z$ is a bijective and
explicit parametrization of all bounded dual windows with sufficiently
small support, i.e., $\supp{h}\subset \itvcc{-\kmax-1}{\kmax+1}$. Of course, the parametrization in
Lemma~\ref{lem:para-all-duals} only works for our setting, in
particular, only under the
2-overlap condition, i.e., $a=L/2$,
while formula~\eqref{eq:all-duals-li} works for all Gabor frames.   

\subsection{Basic properties: symmetry, boundedness, and continuity}
\label{sec:properties-h_z}

The next two lemmas show that the chosen parametrization $z \mapsto h_z$ is rather natural. Indeed, both symmetry and boundedness properties of $g$ and $z$
are transferred to $h_z$.  

\begin{lemma}[Boundedness]
\label{lem:boundedness}
    Let $b \in \itvoo{0}{1}$. Suppose $g \in V_+^0(\R)$. Then $h_z \in L^\infty([0,1])$ if and only
  if $z \in L^\infty(\R)$. 
\end{lemma}
\begin{proof}
To show the ``only if''-implication, note that for $x \in \itvcc{0}{1}$, we have
$h_z(x)=g(x-1)z(x)+b\psi(x)$. Since $\abs{g(x)}$ is bounded (on $\R$) and positive on $\itvoo{-1}{1}$,
it follows that if $h_z$ is bounded on $\itvcc{0}{1}$, then so is $z$
on $\itvcc{c}{1}$ for any $c>0$. Using the boundedness of $h_z$ on
$\itvcc{-1}{0}$ leads to the conclusion that $z$ is bounded on
$\itvcc{0}{1-c}$ for any $c>0$.

For the converse assertion, let $k \in \set{1,\dots,\kmax}$. Once we
argue for the boundedness of $\gamma_k$ and $\eta_k$ on
$\itvcc{k/b}{k+1}$ and $\itvcc{-k-1}{-k/b}$, respectively, the
assertion is clear. We only consider $\gamma_k$ as the argument of
$\eta_k$ is similar. Now, to argue for boundedness of $\gamma_k$,
it suffices to show that the product
\[ 
\prod_{j \in [k]} \frac{1}{g(x-j(1/b-1))} \quad \text{for } x \in \itvcc{k(1/b-1)}{1}
\]
is bounded. For all $j \in [k]$ and all $x \in \itvcc{k(1/b-1)}{1}$ we have
\[ 
  0 = k\left(\frac{1}{b}-1\right) - k\left(\frac{1}{b}-1\right) \le x
  - j\left(\frac{1}{b}-1\right)\le 2 - \frac{1}{b} < 1.  
\]
Let $c$ denote the positive minimum of the continuous function $\abs{g}$ on the compact
interval $\itvcc{-2+1/b}{2-1/b}$. Then  
\[ 
 \sup_{x \in \itvcc{k(1/b-1)}{1}} \prod_{j \in [k]}
 \abs{\frac{1}{g(x-j(1/b-1))}} \le k/c,
\]
which completes the proof. 
\end{proof}

\begin{lemma}[Symmetry]
\label{lem:symmetry}
   Let $b \in \itvoo{0}{1}$. Suppose $g \in V_+^0(\R)$ is even. Then $h_z$ is even if and only if
  $z$ is antisymmetric around
  $x=1/2$, i.e., $z(x)=-z(1-x)$ for a.e. $x \in \itvcc{0}{1/2}$. 
\end{lemma}
\begin{proof}
Assume first $z(x)=-z(1-x)$ for a.e.\ $x \in \R$. If $g$ is even, then
so is $\psi$, and it follows by
  straightforward verification in the definition~\eqref{eq:def-h} that
  $h_z(x)=h_z(-x)$ holds for a.e.\ $x \in \R$.  

On the other hand, if $h_z$ is even, then using the definition of $h_z$ on
$\itvcc{-1}{1}$, it follows easily that $z$ is antisymmetric around
  $x=1/2$. 
\end{proof}

The last lemma of this subsection characterizes continuity of the dual windows $h_z$ in
terms of easy verifiable conditions on $z$.

\begin{lemma}[Continuity]
\label{lem:continuity}
 Let $b \in \itvoo{0}{1}$. Suppose $g \in V^0_+(\R)$. Then $h_z \in C^0(\R)$ if and only if  
$z:\itvcc{0}{1}\to \C$ is a continuous function satisfying
\begin{equation}
  z(0)=\frac{b\,\psi(0)}{g(0)}=\frac{b}{g(0)^2}
\label{eq:C0-cond-on-z-left}
\end{equation}
  and
\begin{equation}
z(1)=-\frac{b\,\psi(1)}{g(0)}=-\frac{b}{g(0)^2}.
\label{eq:C0-cond-on-z-right}
\end{equation}
\end{lemma}
\begin{proof}
Suppose first that $h_z \in C^0(\R)$. Then since for $x \in \itvoc{0}{1}$, we have
$g(x-1)\neq 0$ and $h_z(x)=g(x-1)z(x)+b\psi(x)$ with $h_z ,g,\psi\in C^0(\R)$, it follows that $z$ is continuous on
$\itvoc{0}{1}$. Continuity of $z$ at $x=0$, i.e., existence of
$\lim_{x \searrow 0} z(x)$, follows by similar considerations of
$h_z\vert_{\itvco{-1}{0}}$. By continuity of $h_z$ and
\eqref{eq:support-of-h}, we have $h_z(\pm 1)=0$. Hence, 
\begin{align*}
  0 = h_z(-1) = -g(0)z(0) + b\psi(0) \quad \text{and} \quad  0 = h_z(1) = g(0)z(1) + b\psi(1)
\end{align*}
which shows the ``only if''-implication.

To show the other implication, we have to work a little harder. On the
open set 
\[
\R \setminus \left(\bigcup_{k=0}^{\kmax} \set{\pm k/b,\pm(k+1)} \right),
\]
the function $h_z$ is continuous since it is a sum and product of continuous functions on
this set. 

To show continuity at the seam points $\cup_{k=0}^{\kmax} \set{\pm k/b,\pm(k+1)}$, it
suffices to show
\[
  h_z(0^-) = h_z(0^+), \quad h_z(\pm k/b)=0, \: k\neq 0, \quad
  \text{and} \quad   h_z(\pm (k+1))=0. 
\] 
for $k \in \set{0,1,\dots,\kmax}$. Fix $k \in \set{0,1,\dots,
  \kmax}$.
We first focus on the seam points in $\itvoc{-\infty}{-1}$. For
$x=-k-1$ we immediately have
  \begin{align*}
    h_z(-k-1)&=(-1)^{k} \prod_{j \in [k]}
  \frac{g(j(1/b-1))}{g(-1+j(1/b-1))} \left[ 
-g(0) z(0) + b \psi(0) \right] = 0
  \end{align*}
as the expression in the square brackets is zero by our assumption on
$z(0)$ in \eqref{eq:C0-cond-on-z-left}. 

For $x=-k/b$, $k \neq 0$, we note that $h_z(-k/b)$, by definition, contains the product 
\[
\prod_{j \in [k]}
  \frac{g(-k/b+k+1+j(1/b-1))}{g(-k/b+k+j(1/b-1))} 
\]
as a factor. Further, since $k\in [k]$, this product has
$ \frac{g(-k/b+k+1+k(1/b-1))}{g(-k/b+k+k(1/b-1))}=\frac{g(1)}{g(0)}$
as one of its factors. Since $g(1)=0$ by the support and continuity
assumption $g \in V_+^0(\R)$, it follows that $h_z(-k/b)=0$ for
$k=1,\dots, \kmax$.

Now, we consider seam points in $\itvco{1}{\infty}$. For $x=k+1$ we
have by \eqref{eq:C0-cond-on-z-right} that
\[
h_z(k+1)= (-1)^{k} \prod_{j \in [k]}
\frac{g(-j(1/b-1))}{g(1-j(1/b-1))} \left[ g(0) z(1) + b \psi(1)
\right] = 0.
\]
To see $h_z(k/b)=0$ we note that the product defining $h_z(k/b)$
contains the factor $g(-1)/g(0)$ which is zero due to the assumption
$g \in V_+^0(\R)$.

Finally, we show continuity of $h_z(x)$ at $x=0$. However, this
follows readily by considering the two one-sided limits $x \nearrow 0$
and $x \searrow 0$ of $h_z(x)$:
\begin{align*}
  h_z(0^-)&=-g(1)z(1)+b\,\psi(1)=b \,\psi(1) = b/g(0)
\intertext{and}
  h_z(0^+)&=g(-1)z(0)+b\,\psi(0)=b\,\psi(0) = b/g(0),
\end{align*}
respectively.
%
\end{proof}

\subsection{Higher order smoothness}
\label{sec:high-order-smoothn}

The main result of this section, Theorem~\ref{thm:higher-order-smooth}, characterizes
$C^n$-smoothness of $h_z$ in terms of conditions on $z$.  
 As these conditions involves derivatives of $\psi$, more precisely,
 $\psi^{(m)}(0)$, the next lemma shows how to operate with this
 condition. 
\begin{lemma}
\label{lem:higher-diff-psi-at-zero}
   Let $n \in \Z_{>0}$, and let $g \in V^n_+(\R)$. Then $\psi := \frac{1}{\sum_{n \in \Z} g(\cdot+n)}  \in C^n(\R)$
   and 
   \begin{equation}
\psi^{(n)}(0) = \sum_{\mathbf{m}\in M} \frac{n!}{m_1!m_2!\cdots m_n!}  \frac{(-1)^{m_1+\dots+m_n}(m_1+\dots+m_n)!}{g(0)^{m_1+\dots+m_n+1}} \prod_{j=1}^{n} \left( \frac{g^{(j)}(0)}{j!} \right)^{m_j},
\label{eq:diff-of-psi-at-zero}
\end{equation}
where $M:=\{(m_1,\dots,m_n)\in(\Z_{\ge 0})^n |1 \cdot m_1+2 \cdot m_2+\dots+n \cdot m_n=n\}$.
\end{lemma}
\begin{proof}
Since $g \in V_+^n(\R)$, the sum $\sum_{n \in \Z} g(\cdot+n)$ is
in $C^n(\R)$ and is bounded below by a positive constant. Hence, since
the mapping $x\mapsto 1/x$ is $C^\infty$ on $\itvoo{0}{\infty}$, we
have $\psi \in C^n(\R)$. 

As usual, we consider $\psi$ on $\itvcc{0}{1}$ where the function is given by
$x\mapsto 1/(g(x)+g(x-1))$. We will use the following version of Fa\`a di Bruno's formula: 
\begin{equation} 
\label{eq:FaaDiBruno}
\frac{d^n}{d x^n}f\left(h(x)\right)=
\sum_{\mathbf{m}\in M}
\frac{n!}{m_1!m_2!\cdots m_n!} 
f^{(m_1+\cdots+m_n)}\left(h(x)\right)
\prod_{j=1}^{n}\left(\frac{h^{(j)}(x)}{j!}\right)^{m_j}.
\end{equation}
Taking $f(x)=\frac{1}{x}$ and $h(x)=g(x)+g(x-1)$ in Fa\`a di Bruno's formula yields
\begin{align*}
\psi^{(n)}(0) &= \sum_{\mathbf{m}\in M} \frac{n!}{m_1!m_2!\cdots m_n!} \frac{(-1)^{m_1+\dots+m_n}(m_1+\dots+m_n)!}{(g(0)+g(-1))^{m_1+\dots+m_n+1}} \prod_{j=1}^{n} \left( \frac{g^{(j)}(0)+g^{(j)}(-1)}{j!} \right)^{m_j}\\
&= \sum_{\mathbf{m}\in M} \frac{n!}{m_1!m_2!\cdots m_n!} \frac{(-1)^{m_1+\dots+m_n}(m_1+\dots+m_n)!}{g(0)^{m_1+\dots+m_n+1}} \prod_{j=1}^{n} \left( \frac{g^{(j)}(0)}{j!} \right)^{m_j},
\end{align*}
using $d^m(x^{-1})/dx^m = (-1)^m m! x^{-(m+1)}$ for $m
\in \Z_{>0}$ in the first equality and $g^{(\ell)}(-1)=0$ for
$\ell=0,1,\dots,n$ in the second.
\end{proof}

\begin{example}
\label{ex:compute-psi-deri-at-0}
  We illustrate the computation of
  Lemma~\ref{lem:higher-diff-psi-at-zero} for $n=1,2$. For $n=1$,
  since $M=\set{1}$, formula (\ref{eq:diff-of-psi-at-zero}) simply
  yields
  \begin{equation*}
  \psi^{(1)}(0) = -  \frac{g^{(1)}(0)}{g(0)^{2}}.
\end{equation*} 
 For $n=2$, we have $M=\set{(2,0),(0,1)}$, whereby
       \begin{equation*}
\psi^{(2)}(0) = 2\frac{g^{(1)}(0)^2}{g(0)^{3}} - \frac{g^{(2)}(0)}{g(0)^{2}}.
\end{equation*}
\end{example}

\begin{theorem}
\label{thm:higher-order-smooth}
  Let $n \in \Z_{>0}\cup \set{\infty}$, and let $g \in V^n_+(\R)$. The following assertions
  are equivalent: 
  \begin{enumerate}[(i)]
  \item $z\in C^n(\itvcc{0}{1})$ satisfies
    \eqref{eq:C0-cond-on-z-left}, \eqref{eq:C0-cond-on-z-right}, and,
    for each $m=1,\dots, n,$
    \begin{align}
      \label{eq:Cn-cond-on-z-left}
      z^{(m)}(0) &= -\sum_{\ell=1}^m  \binom{m}{\ell}
                   \frac{g^{(\ell)}(0)}{g(0)}  z^{(m-\ell)}(0) + b\frac{\psi^{(m)}(0)}{g(0)} ,
                   \intertext{and}
                   z^{(m)}(1) &= -\sum_{\ell=1}^m  \binom{m}{\ell}
                                \frac{g^{(\ell)}(0)}{g(0)}  z^{(m-\ell)}(1) - b\frac{\psi^{(m)}(0)}{g(0)}.
                                \label{eq:Cn-cond-on-z-right}
    \end{align}
\item $h_z \in C^n(\R)$.
  \end{enumerate}

\end{theorem}
\begin{proof}
We first prove the assertion (i)$\Rightarrow$(ii). Consider the open set
\[
J:=\R \setminus \left(\bigcup_{k=0}^{\kmax} \set{\pm k/b,\pm(k+1)} \right).
\]
The function $h_z\vert_{J}$ is in $C^{n}(J)$ since it is a sum and product of $C^{n}(J)$ functions.

To prove $C^n$-smoothness at the seam points $\cup_{k=0}^{\kmax}
\set{\pm k/b,\pm(k+1)}$, we need to show that 
\[
  h^{(m)}_z(0^-) = h^{(m)}_z(0^+), \quad h^{(m)}_z(\pm k/b)=0, \: k\neq 0, \quad
  \text{and} \quad   h^{(m)}_z(\pm (k+1))=0
\] 
for $m=1,\dots,n$ and $k \in \set{0,1,\dots,\kmax}$.
The proof of this is split into the three cases:
\[x=0, \quad x\in\bigcup_{k=0}^{\kmax} \set{\pm k/b} \quad \text{and} \quad x\in\bigcup_{k=0}^{\kmax} \set{\pm(k+1)}.\]
However, we first note that $g \in V_+^n(\R)$ implies $g^{(m)}(\pm 1)=0$ for $m=0,\dots,n$,

\emph{$C^n$-smoothness at $x=0$.}
For $x \in \itvcc{0}{1}$ 
\begin{align*}
h_z^{(m)}(x) =\sum_{\ell=0}^{m} \binom{m}{\ell}
  g^{(\ell)}(x+1)z^{(m-\ell)}(x+1)+ b\psi^{(m)}(x).
\end{align*}
Since $g^{(m)}(1)=0$ for $m=0,\dots,n$, it is readily seen that 
\begin{align*}
h_z^{(m)}(0^-)&=\sum_{\ell=0}^{m} \binom{m}{\ell}
                g^{(\ell)}(1)z^{(m-\ell)}(1) + b\psi^{(m)}(0) \\ &= b\psi^{(m)}(0) 
\end{align*}
for all $m \le n$.

Similarly, by considering $x \in \itvco{-1}{0}$, we see that 
\begin{align*}
h_z^{(m)}(0^+)&=-\sum_{\ell=0}^{m} \binom{m}{\ell}
                g^{(\ell)}(-1)z^{(m-\ell)}(0) + b\psi^{(m)}(1) = b\psi^{(m)}(0),
\end{align*}
where the last equality follows from $g^{(m)}(-1)=0$ for
$m=0,\dots,n$ and periodicity of $\psi$.

In the two remaining cases, we will use the following easy consequence of the Leibniz rule~\eqref{eq:leibniz}. If $f,g \in C^n(\R)$ and
$g^{(m)}(x_0)=0$ for $m=0,1, \dots, n$, then $(fg)^{(m)}(x_0)=0$ for
$m=0,1, \dots, n$, where $fg=x \mapsto f(x)g(x)$.

\emph{$C^n$-smoothness at $x=\pm k/b$, $k=1,\dots,\kmax$.}
We first consider $x=k/b$, $k=1,\dots,\kmax$. By rearranging the terms
in the definition of the auxiliary function $\gamma_k(x+k)$  in
(\ref{def_gamma}), we obtain 
\begin{multline*}
  \gamma_{k}(x+k) =g(x-1-k(1/b-1)) \\ \cdot \frac{(-1)^{k}}{g(x-k(1/b-1)))}\prod_{j \in [k-1]}
  \frac{g(x-1-j(1/b-1))}{g(x-j(1/b-1))} \left[ 
g(x-1) z(x) + b\psi(x) \right],
\end{multline*}
for $x \in \itvcc{k(1/b-1)}{1}$. Thus, for $x \in \itvcc{k/b}{k+1}$,
we will consider $\gamma_{k}(x)$ a product of the functions
\[
g(x-1-k/b)
\]
and
\[
\frac{(-1)^{k}}{g(x-k/b)}\prod_{j \in [k-1]}
  \frac{g(x-k-1-j(1/b-1))}{g(x-k-j(1/b-1))} \left[ 
g(x-k-1) z(x-k) + b\psi(x-k) \right].
\]
Focusing on the derivatives of $g(x-1-k/b))$ at $x=k/b$, we get
\[
D^m\left[g(x-1-k/b)\right](k/b)=g^{(m)}(-1)=0
\]
for $m=0,1,\dots,n$. Hence, as a consequence of the Leibniz rule~\eqref{eq:leibniz},
\[
\gamma_k^{(m)}(k/b)=0, \quad \text{for $m=1,2,\dots,n$}.
\]
 The calculations for $x=-k/b$, $k=1,\dots,\kmax$ are similar to the ones above hence we leave these to the reader.

 \emph{$C^n$-smoothness at $x=\pm(k+1)$, $k=1,\dots,\kmax$.}  Again,
 we first consider $x=k+1$, $k=1,\dots,\kmax$. For
 $x \in \itvcc{k/b}{k+1}$, we consider $\gamma_k(x)$ as a product of
 the two functions
\[
\prod_{j \in [k]}
  \frac{g(x-k-1-j(1/b-1))}{g(x-k-j(1/b-1))}
\]
and
\[
g(x-k-1) z(x-k) + b\psi(x-k).
\]
 We focus on the second of the two and observe that, for each $m=1,2,\dots,n$,
\begin{multline*}
D^m\left[g(\cdot-k-1) z(\cdot-k) + b\psi(\cdot-k)\right](k+1) =
\sum_{\ell=0}^m  \binom{m}{\ell}g^{(\ell)}(0)  z^{(m-\ell)}(1) +
b\psi^{(m)}(0) \\
=  g(0) \left[ z^{(m)}(1) +\sum_{\ell=1}^m  \binom{m}{\ell}\frac{g^{(\ell)}(0)}{g(0)} z^{(m-\ell)}(1) +
  b\frac{\psi^{(m)}(0)}{g(0)} \right] 
= 0,
\end{multline*}
where the final equality follows from the assumption \eqref{eq:Cn-cond-on-z-right}.
Thus, as a consequence of the Leibniz rule~\eqref{eq:leibniz}, we arrive at
\[\gamma_k^{(m)}(k+1)=0,\]
for $m=1,2,\dots,n$. The calculations for $x=-k-1$, $k=1,\dots,\kmax$ are similar to the ones above, hence we leave these to the reader.

The proof of assertion (ii)$\Rightarrow$(i) is similar to the argument
in the proof of Lemma~\ref{lem:continuity}, hence we will omit the proof. 
\end{proof}

\begin{example}
\label{ex:comp-higher-order-cond}
We compute \eqref{eq:Cn-cond-on-z-left} and
  \eqref{eq:Cn-cond-on-z-right} for $n=2$ using
  Example~\ref{ex:compute-psi-deri-at-0}. For $m=1$, using that 
$z$ also has to satisfy (\ref{eq:C0-cond-on-z-left}) and
(\ref{eq:C0-cond-on-z-right}), we get 
\begin{align}
      \label{eq:z1-smooth}
      z^{(1)}(0) &= -z^{(1)}(1) = -2b\frac{g^{(1)}(0) }{g(0)^{3}}.
\intertext{Similarly, for $m=2$, using that $z$ satisfies
(\ref{eq:C0-cond-on-z-left}), (\ref{eq:C0-cond-on-z-right}),
 and (\ref{eq:z1-smooth}), we get}
      \label{eq:z2-smooth}
      z^{(2)}(0) &= -z^{(2)}(1) = 6 b \frac{g^{(1)}(0)^2}{g(0)^{4}} - 2b \frac{g^{(2)}(0)}{g(0)^{3}} .
\end{align}
\end{example}

It is easy to find a function $z:\itvcc{0}{1}\to \C$ satisfying the
conditions in (i) in Theorem~\ref{thm:higher-order-smooth}. E.g., if
$n < \infty$, a polynomial $z$ of degree $2n+2$ will always
do. Further, as the conditions on $z \in C^n(\R)$ only concern
the derivatives of $z(x)$ at the boundary points $x = 0$ and $x=1$, there
is an abundance of $C^n$ dual windows of each $g \in V^n_+(\R)$ for
any value of $b \in \itvoo{0}{1}$.  Hence, given a window
$g \in V^n_+(\R)$, we can easily construct dual windows in $C^n(\R)$
using Theorem~\ref{thm:higher-order-smooth}; example of such
constructions will be given in Section~\ref{sec:examples}.

We now exhibit a large class of window functions $g \in V_+^n(\R)$ containing,
e.g.,  all symmetric windows and all windows forming a
partition of unity, for
which the issue of computing $\psi^{(m)}(0)$ used in
Theorem~\ref{thm:higher-order-smooth} disappears. 
\begin{corollary}
\label{cor:higher-order-smooth-part-of-unity}
    Let $n \in \Z_{>0} \cup \set{\infty}$. Suppose $g \in V^n_+(\R)$
    satisfies $g^{(m)}(0)=0$ for $m =1,\dots,n$. Then the following assertions are
    equivalent:
    \begin{enumerate}[(i)] 
\item $z\in C^n(\itvcc{0}{1})$ satisfies
      \eqref{eq:C0-cond-on-z-left}, \eqref{eq:C0-cond-on-z-right},
      and, for each $m=1,\dots, n,$
      \begin{align}
        \label{eq:Cn-cond-on-z-left-pou}
        z^{(m)}(0) &= 0, 
                     \intertext{and}
                     z^{(m)}(1) &=  0, 
                                  \label{eq:Cn-cond-on-z-right-pou}
      \end{align}
    \item $h_z \in C^n(\R)$.
    \end{enumerate}
In particular, if \begin{equation}
    z(x) = \frac{b}{g(0)^3} \bigl[ 2 g(x)-g(0)\bigr],\label{eq:z-standard}
  \end{equation}
then $h_z \in C^n(\R)$.
\end{corollary}
\begin{proof}
 By Lemma~\ref{lem:higher-diff-psi-at-zero}, it follows that $\psi^{(m)}(0)=0$
for  $m = 1,\dots,n$. With $g^{(m)}(0)=0$
and $\psi^{(m)}(0)=0$ for all $m =1,\dots,n$ conditions~\eqref{eq:Cn-cond-on-z-left} and
\eqref{eq:Cn-cond-on-z-right} reduce to
\eqref{eq:Cn-cond-on-z-left-pou} and
\eqref{eq:Cn-cond-on-z-right-pou}, respectively. 
Finally, it is straightforward to verify that
$z$ defined by \eqref{eq:z-standard} satisfies the $2n+2$ conditions in (i).
\end{proof}


\begin{remark}
\label{rem:cor-on-smoothness}
  \begin{enumerate}[(a)]
  \item Suppose $g\in V_+^n(\R)$ satisfies the assumptions of
    Corollary~\ref{cor:higher-order-smooth-part-of-unity}, i.e.,
    $g^{(m)}(0)=0$ for $m =1,\dots,n$. Now, if $h\in C^n(\R)$ and $b \in \C$ satisfy the window condition
  $\sum_{n\in \Z} g(x+n)\overline{h(x+n)}=b$ for $x \in \R$, then by
  term-wise differentiating the window condition we see that
  $h^{(m)}(0)=0$ for $m=1,\dots,n$. Thus, any dual 
  window in $C^n(\R)$, not necessarily with compact support, will also have this
  property. 
\item 
    Suppose $g \in V^n_+(\R)$ satisfies either
    $\sum_{n \in \Z} g(x+n)=1$ or $g(x)=g(-x)$ for $x \in \R$ (or
    both).  Then $g$ satisfies the assumptions of
    Corollary~\ref{cor:higher-order-smooth-part-of-unity}, i.e.,
    $g^{(m)}(0)=0$ for $m =1,\dots,n$. For $g$ symmetric, this is
    obvious.  If $g\in V^n_+(\R)$ forms a partition of unity, then, by
    differentiating $g(x)+g(x-1)=1$ for $x \in \itvcc{0}{1}$, one
    will see that $g^{(m)}(0)=0$ for $m=1,\dots,n$.
  \item The dual window $h_z$ defined by \eqref{eq:z-standard} in
    Corollary~\ref{cor:higher-order-smooth-part-of-unity} is often a
    convenient choice as it guarantees that the dual window $h_z$ is
    defined only in terms of the window $g$. Hence, if $g$ is, e.g., a
    piecewise polynomial, then $h_z$ becomes a piecewise rational
    function of polynomials. However, if $g$ is symmetric, $h_z$
    defined by \eqref{eq:z-standard} will only be symmetric, if
    $g(x)=g(0)-g(1-x)$ for $x \in \itvcc{0}{1}$, that is, if the graph
    of $g$, restricted to $\itvcc{0}{1}\times \C$, is anti-symmetric around $(1/2,g(0)/2)$.
  \end{enumerate}
\end{remark}
\subsection{Optimality of smoothness}
\label{sec:optimality}

 The first result of this section shows
that, even though there are an abundance of dual windows in $C^n(\R)$ for
$g \in V^n_+(\R) \setminus C^{n+1}(\R)$, additional smoothness, e.g., dual windows in $C^{n+1}(\R)$, is in
general not possible. 

\begin{proposition}
  \label{thm:Cn-is-optimal-comp-h}
  Let $b \in \itvoo{0}{1}$. Let $n \in \Z_{\ge 0}$, and let
  $g \in V^n_+(\R)$ with $\sum_{n \in \Z} g(x+n)=1$ for $x \in
  \R$.
  Assume $g$ is a real-valued, piecewise $C^{n+1}$-function for which
  $g^{(n+1)}$ has a simple discontinuity at $x = -1$, $x=0$, \emph{and/or}
  $x=1$.
  If $h\in L^2(\R)$ is compactly supported in $\itvcc{-\kmax-1}{\kmax+1}$, and $\gaborA$ and
  $\gaborA[h]$ are dual frames, then $h \notin C^{n+1}(\R)$.
\end{proposition}
\begin{proof}
  Assume towards a contradiction that $h=h_z \in
  C^{n+1}(\R)$. Recall that $g \in V^n_+(\R)$ implies 
  $g^{(m)}(\pm 1)=0$ for $m=0,\dots,n$. We only consider the case
  $g(x)>0$ for $x \in \itvoo{-1}{1}$ as the argument for $g(x)<0$ is similar. Since $g(x)>0$ for $x \in \itvoo{-1}{1}$, it follows that
  $g^{(n+1)}(-1^+) \ge 0$ and $g^{(n+1)}(1^-) \le 0$. Assume that $g^{(n+1)}$ is discontinuous at $x = -1$ and/or
  $x=1$, i.e., $g^{(n+1)}(-1^+) > 0$ and/or $g^{(n+1)}(1^-) < 0$. The case $x=0$ follows from
  Theorem~\ref{prp:Cn-is-optimal-general-h-frame} below. 
  
  As in the proof of Theorem~\ref{thm:higher-order-smooth}, we see that
\[
h_z^{(n+1)}(0^-) =  g^{n+1}(1^-)z(1) = -g^{n+1}(1^-)b/g(0)^2 
\]
and 
\[
h_z^{(n+1)}(0^+) =- g^{n+1}(-1^+)z(0) = -g^{n+1}(-1^+)b/g(0)^2 
\]
Since, by assumption, $h_z^{(n+1)}(0^-)=h_z^{(n+1)}(0^+)$, it follows that
$g^{n+1}(1^-)=g^{n+1}(-1^+)$, which is a contradiction. 
\end{proof}

\begin{example}
\label{ex:bspine-B2-no-C1-dual}
Let $g = \max(0,1-\abs{x})$ be the second cardinal B-spline with
uniform knots. For any $b \in \itvoo{0}{1}$, the Gabor system
$\gaborA$ is a frame for $L^2(\R)$. Since $g'(x)$ is discontinuous at
$x=-1$ (and at $x=0$ and $x=1$), it follows by
Proposition~\ref{thm:Cn-is-optimal-comp-h} that no dual
windows $h \in C^1(\R)\cap L^2(\R)$ with support in
$\itvcc{-\kmax-1}{\kmax+1}$ exists for any
value of $b$.
\end{example}

Let us comment on the assumptions of
Proposition~\ref{thm:Cn-is-optimal-comp-h}. The location of the
discontinuity of $g^{(n+1)}$ is important. In fact, if discontinuities
of $g^{(n+1)}$ avoid certain points, the conclusion
$h \notin C^{n+1}(\R)$ of Proposition~\ref{thm:Cn-is-optimal-comp-h}
may not hold. On the other hand, we assume the partition of unity property of $g$
only for convenience as to simplify the proof. Furthermore, as we see
by the next two results, positivity of $\abs{g(x)}$ on $\itvoo{-1}{1}$
and compact support of the dual window $h$ are also not essential for
obstructions results on the smoothness of dual windows. For $n=-1$, we ignore the
 requirement $g \in C^n(\R)$ in the formulation below.
\begin{lemma}
  \label{lem:general-obstruc-window-cond}
 Let $n \in \Z_{\ge -1}$, and let $g \in
 C^n(\R)$ be a piecewise $C^{n+1}$-function. Let
 $\{x_r\}_{r\in [R]}=\{x_1,\dots,x_R\}$ denote the finite set of points, where $g^{(n+1)}$
 has simple discontinuities. Assume $h
 \in C^{n+1}(\R)$ and the constants $a>0$, $b \in \C$ satisfy the window condition 
 \begin{equation}
\sum_{n \in \Z} g(x+an)\overline{h(x+an)} = b,  \quad \text{for
  all } x \in \itvcc{-a/2}{a/2}.
\label{eq:general-windows-central-cond}
\end{equation}
Let $\set{t_r}_{r \in [R]}=\{x_r\}_{r\in [R]} \ (\bmod \ a)$ so
that $t_r \in \itvco{-a/2}{a/2}$ and $t_r \le t_{r+1}$. Set $t_0=-a/2$ and
$t_{R+1}=a/2$. Suppose $\sum_{n \in \Z} D^m[g(\cdot+an)\overline{h(\cdot+an)}]$
converges uniformly on $\itvcc{t_r}{t_{r+1}}$ for $m = 1,\dots,n+1$
and $r=0,\dots, R$.
Then, for each $r=1,\dots,R$,
\begin{equation}
  \sum_{\{s \in [R]: x_s-x_r \in a\Z\} } \left[
    g^{(n+1)}(x_s^+)-g^{(n+1)}(x_s^-)\right] h(x_r) =
  0.
\label{eq:restric-on-h-at-discon}
\end{equation}
In particular, if $\{s \in [R]: x_s-x_r \in a\Z\} = \{r\}$, then $h(x_r) = 0$. 
\end{lemma}
\begin{proof}
  The following argument is inspired by the proof of Lemma~1 in
  \cite{LaugesenGabor2009}. Fix $r \in [R]$. Uniform convergence of
  $\sum_{n \in \Z} D^m[g(\cdot+an)\overline{h(\cdot+an)}]$ for each
  $m = 0,1,\dots,n+1$ allows us to differentiate the window condition
  \eqref{eq:general-windows-central-cond} term by term \cite[Theorem~7.17]{RudinPrinciples1976}. Differentiating
  $m$ times then gives:
  \begin{equation}
  \sum_{n \in \Z_{>0}} \sum_{\ell =0}^m \binom{m}{\ell} g^{(\ell)}(x+an)\overline{h^{(m-\ell)}(x+an)} = 0,  \quad \text{for
  all } x \in \itvoo{t_r}{t_{r+1}}.
\label{eq:general-windows-term-wise-diff}
\end{equation}
for all $r =0,1,\dots , R$ and $m=1,\dots, n$.
Note that the sum \eqref{eq:general-windows-term-wise-diff} is $a$
periodic. Hence, by subtracting the two one-sided limits $x \nearrow t_r$
and $x \searrow t_r$ of \eqref{eq:general-windows-term-wise-diff}, we
obtain \eqref{eq:restric-on-h-at-discon}.
\end{proof}

Recall that if $\gabor$ and $\gabor[h]$ with $g,h \in L^2(\R)$ are dual
  frames for $L^2(\R)$, then \eqref{eq:general-windows-central-cond}
  holds. Hence, under the assumptions of
  Lemma~\ref{lem:general-obstruc-window-cond}, duality of $g$ and $h$ restricts the
  possible values of $h$ on $\{x_j\}_{j \in J}$. For windows $g$ with
  support in $\itvcc{-1}{1}$,
  Lemma~\ref{lem:general-obstruc-window-cond} leads to the following
  general obstruction result.
\begin{theorem}
\label{prp:Cn-is-optimal-general-h-frame}
  Let $b \in \itvoo{0}{1}$ and $h \in L^2(\R)$.  Let $n \in \Z_{\ge -1}$, and let $g \in
 C^n(\R)$ be a piecewise $C^{n+1}$-function with
 $\supp g \subset \itvcc{-1}{1}$, and let
 $\{x_j\}_{j\in J} \subset \itvcc{-1}{1}$ denote the finite set of points, where $g^{(n+1)}$
 is discontinuous. Assume either 
 \begin{enumerate}[(i)]
 \item $0 \in \{x_j\}_{j\in J}$ and $\supp h \subset \itvcc{-\kmax-1}{\kmax+1}$, or
 \item $0 \in \{x_j\}_{j\in J}$ and $\pm 1 \notin \{x_j\}_{j\in J}$.
 \end{enumerate}
 If $\gaborA$ and $\gaborA[h]$ are dual
  frames, then $h \notin C^{n+1}(\R)$.
\end{theorem}
\begin{proof}
  Assume towards a contradiction that $h \in C^{n+1}(\R)$.
  From $\supp g \subset [-1,1]$, it follows that $g(k)=0$ for all
  $k \in \Z\setminus \{0\}$ and that
  $\set{x_j}_{j \in J} \subset \itvcc{-1}{1}$. Hence, equation
  \eqref{eq:char-eqns} for $k=0$ implies that
  $h(0)=b/g(0)>0$. Depending on whether we use assumption (i) or
  (ii), the
  points $x=\pm 1$ may or may not belong to $\set{x_j}_{j \in J}$. In
  either case, $g^{(n+1)}(-1^-)=g^{(n+1)}(1^+)=0$ and we have from \eqref{eq:restric-on-h-at-discon} that
\[ 
   0 =    \left[g^{(n+1)}(-1^+)-0)\right] h(-1) +   \left[
    g^{(n+1)}(0^+)-g^{(n+1)}(0^-)\right] h(0) +
  \left[0-g^{(n+1)}(1^-)\right] h(1) 
\]
If we use assumption (i), then, by the compact support of $h$, we have
from \cite[Lemma~3.2]{ChristensenGabor2010} that $h(\pm 1)=0$. On
the other hand, from assumption (ii), we have
$g^{(n+1)}(-1^+)=g^{(n+1)}(1^-)=0$. In either case, we get
  \[ 
  \left[g^{(n+1)}(0^+)-g^{(n+1)}(0^-) \right] h(0) = 0,
  \]
  which is a contradiction to $0 \in \{x_j\}_{j\in J}$ and $h(0)>0$.
\end{proof}

The conditions on the window $g$ in the above results should be understood as
follows. In order to make the statement as strong as possible, we want
generators $g$ just shy of being in the $C^{n+1}$-class. Hence, the
function $g$
is assumed to be $C^n$ everywhere and piecewise $C^{n+1}$ except at a finite
number of points, where both one-sided limits of $g^{(n+1)}$ exist, but
do not agree. The following example shows a general, but typical,
obstruction of the smoothness of dual windows.

\begin{example}
\label{ex:no-smooth-dual}
Let $b \in \itvoo{0}{1}$, $n \in \Z_{>0}$, and let $g\in C^n(\R)$ with
$\supp g \subset \itvcc{-1}{1}$ be a
$C^\infty$-function except at $x=0$, where $g^{(n+1)}$ fails to be
continuous. Suppose the Gabor system $\gaborA$ is a frame for
$L^2(\R)$. Then, by
Theorem~\ref{prp:Cn-is-optimal-general-h-frame}, it follows that
no dual windows $h \in C^{n+1}(\R)\cap L^2(\R)$ exists. Note that this
conclusion holds whether or not $h$ is assumed to have compact support.
\end{example}

\subsection{Small support}
\label{sec:duals-with-small}

While the previous section considered optimality of the smoothness of
the dual windows, we are here concerned with optimizing, that is, minimizing, the support
length. Such questions were considered in \cite{ChristensenGabor2012},
where the authors characterized the existence of continuous dual
windows with short support for continuous windows $g$ with finitely
many zeros inside their support $\itvcc{-1}{1}$. In the following
result we consider the possibility of higher
order smoothness of dual windows with short support.

\begin{theorem}
\label{thm:small-support}
Let $n \in \Z_{\ge 0}$, let $b\in\itvcos{\frac{N}{N+1}}{\frac{2N}{2N+1}}$ for some $N\in\Z_{>0}$,
and let $g \in V^n_+(\R)$.
Define
\begin{equation}\label{eq:def_z_small_supp}
z(x)=\begin{cases}
b\frac{\psi(x)}{g(x)} & x\in\itvcc{0}{1-N(\tfrac{1}{b}-1)},\\
z_{\mathit{mid}}(x) & x\in\itvoo{1-N(\tfrac{1}{b}-1)}{N(\tfrac{1}{b}-1)},\\
-b\frac{\psi(x)}{g(x-1)} & x\in\itvcc{N(\tfrac{1}{b}-1)}{1},\\
\end{cases}
\end{equation}
where $z_{\mathit{mid}}:\itvoo{1-N(\tfrac{1}{b}-1)}{N(\tfrac{1}{b}-1)}\to\C$ is a
measurable function. The following assertions hold:
\begin{enumerate}[(a)]
\item 
The dual window $h_z$ has compact support in 
  $\itvcc{-N}{N}$. \label{item:small-supp}
\item $h_z\in C^n(\R)$ if and only if $z\in C^n(\itvcc{0}{1})$. \label{item:small-smooth}
\item Suppose $g$ is even. Then $h_z$ is even if and only if $z_{\mathit{mid}}$ is antisymmetric around $x=1/2$, i.e., $z_{\mathit{mid}}(x)=-z_{\mathit{mid}}(1-x)$ for a.e. $x\in\itvoc{1-N\left(\frac{1}{b}-1\right)}{1/2}$. \label{item:small-symm}
\end{enumerate}
\end{theorem}

\begin{proof}
 By definition of $N$, we have 
  $N(\tfrac{1}{b}-1)\le 1$, hence, $z$ is well-defined. 
 
 \eqref{item:small-supp}:  If $\kmax < N$, then
  $\supp{h_z}\subset \itvcc{-\kmax-1}{\kmax+1} \subset \itvcc{-N}{N} $
  by \eqref{eq:support-of-h}. Assume now that $\kmax\geq N$ and consider
  the dual windows $h_z$ on $\itvcc{N}{\kmax+1}$. For $x>N$, it
  suffices to show that  $h_z(x) = 0$ for
  $x \in \itvoo{k/b}{k+1}$ and $k= N,\dots,\kmax$. Recall
  that, for any $k=1,\dots, \kmax$, 
\[  h_z(x)=\gamma_{k}(x)\]
for $x\in\itvoo{k/b}{k+1}$, which can be rewritten as
\[  h_z(x+k)=\gamma_{k}(x+k) = (-1)^{k} \prod_{j \in [k]}
  \frac{g(x-1-j(1/b-1))}{g(x-j(1/b-1))} \left[ 
g(x-1) z(x) + b\psi(x) \right]\]
for $x\in\itvoo{k/b-k}{1}$.
Since $k\geq N$, we have the inclusion $\itvoo{k(\tfrac{1}{b}-1)}{1}\subset\itvoo{N(\tfrac{1}{b}-1)}{1}$. Hence
\[
g(x-1) z(x) + b\psi(x)=g(x-1)\left(-b\frac{\psi(x)}{g(x-1)}\right)+b\psi(x)=0.
\]
Thus $h_z(x)=0$ for $x>N$. The argument for $x<-N$ is similar so we
omit it. 

\eqref{item:small-smooth}: By inserting $x=0$ and $x=1$ into
\eqref{eq:def_z_small_supp}, it can easily be seen that $z$ satisfies
\eqref{eq:C0-cond-on-z-left} and \eqref{eq:C0-cond-on-z-right},
respectively. Therefore, the result for $n=0$ simply follows from
Lemma~\ref{lem:continuity}. For $n>0$, the ``only if''-assertion
follows directly from Theorem~\ref{thm:higher-order-smooth}. To prove
the other direction, we assume that $z\in C^n(\itvcc{0}{1})$. From
definition \eqref{eq:def_z_small_supp} we have:
\[
g(x)z(x)-b\psi(x)=0 \quad \text{for all }
x\in\itvcc{0}{1-N(\tfrac{1}{b}-1)}.
\]
By differentiating both sides $m$ times, isolating $z^{(m)}(x)$ and
inserting $x=0$, we see that $z$ satisfies
\eqref{eq:Cn-cond-on-z-left}. In a similar way, one proves that $z$ satisfies \eqref{eq:Cn-cond-on-z-right}. Hence, by Theorem~\ref{thm:higher-order-smooth}, $h_z\in C^n(\R)$.

 \eqref{item:small-symm}: From \eqref{eq:def_z_small_supp} we see that, for $x\in\itvcc{0}{1-N\left(\frac{1}{b}-1\right)}$, the function $z$ satisfies
 \[-z(1-x)=b\frac{\psi(1-x)}{g(-x)}=b\frac{\psi(x)}{g(x)}=z(x),\]
 where the second equality uses that $g$ and $\psi$ are even
 and that $\psi$ is $1$-periodic. Hence, if $g$ is even,
 then $z_{\mathit{mid}}$ is antisymmetric around $x=1/2$ if and only if $z$
 defined by \eqref{eq:def_z_small_supp} is
 antisymmetric around $x=1/2$. The conclusion now follows from Lemma~\ref{lem:symmetry}.
\end{proof}

The short support of Theorem~\ref{thm:small-support} is optimal in
the following sense of \cite{ChristensenGabor2012}: If a dual window $h$
with support $\supp{h} \subset \itvcc{-N}{N}$ exists, then necessarily
$b \le 2N/(2N+1)$, see \cite[Theorem~2.3]{ChristensenGabor2012}. If, in addition,
$h$ is assumed to be continuous, then $b<2N/(2N+1)$, see \cite[Theorem~2.5]{ChristensenGabor2012}.


\section{Examples of the construction}
\label{sec:examples}


In this section, we present two examples of the construction procedure
of dual windows using the results from the previous sections. In
Example~\ref{ex:example-hann} we construct dual windows of the
classical and widely used Hann and Blackman window, respectively. In Example~\ref{ex:example-oscill} we
consider a smoother, but non-symmetric window; the setup is more
complicated than Example~\ref{ex:example-hann} and perhaps less useful for
applications, but it serves as a proof of concept of the flexibility of
our method.

\begin{example}
\label{ex:example-hann}
The Hann window $g_{\mathit{hann}}\in C^1(\R)$ is defined by
\begin{equation*}
g_{\mathit{hann}}(x)= \cos^2(\pi x/2)\chi_{\itvcc{-1}{1}}(x) = \begin{cases}
\frac12-\frac12 \cos(\pi(x+1)) & x \in \itvco{-1}{0}\\
\frac12 + \frac12 \cos(\pi(x)) & x \in \itvcc{0}{1}\\
0 & \text{otherwise,}
\end{cases}
\end{equation*}
and the Blackman window $g_{\mathit{blac}}\in C^1(\R)$ is defined by
\[
g_{\mathit{blac}}(x) = \bigl[ 0.42 + 0.5 \cos(\pi x) + 0.08 \cos(2 \pi
  x) \bigr] \chi_{\itvcc{-1}{1}}(x) 
\]
for $x \in \R$, see Figure~\ref{fig:ex1-windows-and-FT}. Both these widows are continuously differentiable,
symmetric, non-negative, and normalized $g(0)=1$, and the
Hann window even has the partition of unity property. Both of the
windows belong to $V^1_+(\R)$, but not $V_+^2(\R)$; hence, the optimal
smoothness of compactly supported dual windows are $h \in C^1(\R)$.
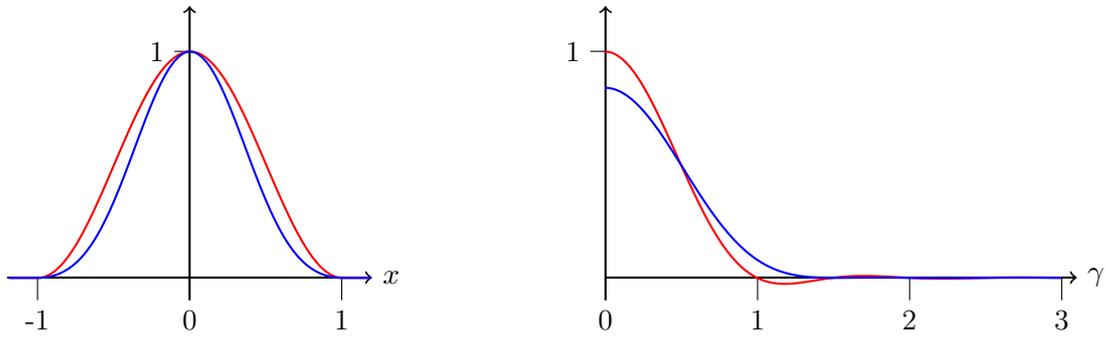
\begin{figure}[!h]
{\centering
    \begin{minipage}[t][][b]{0.4\linewidth}
      \begin{tikzpicture}[scale=2,y=1.5cm,x=1cm]
        \draw[->, thick] (-1.2,0) -- (1.2,0) node[right] {$x$};
        \draw[->, thick] (0,-0.1) -- (0,1.2); \foreach \x in {-1,0,1}
        \draw (\x,0) -- (\x,-0.1) node[below] {\x}; \foreach \y in {1}
        \draw (0,\y) -- (-0.1,\y) node[left] {\y};
        \begin{axis}[mark size=1.25pt,anchor=origin,hide
          axis,y=1.5cm,x=1cm]
          \addplot[color=red] table{data-g-hann.tsv};
          \addplot[color=blue] table{data-g-blackman.tsv};
        \end{axis}
      \end{tikzpicture}
    \end{minipage}
    \hfill
    \begin{minipage}[t][][b]{0.5\linewidth}
      \begin{tikzpicture}[scale=2,y=1.5cm,x=1cm]
        \draw[->, thick] (0,0) -- (3.1,0) node[right] {$\gamma$};
        \draw[->, thick] (0,-0.1) -- (0,1.2); \foreach \x in {0,1,2,3}
        \draw (\x,0) -- (\x,-0.1) node[below] {\x}; \foreach \y in {1}
        \draw (0,\y) -- (-0.1,\y) node[left] {\y};
        \begin{axis}[mark size=1.25pt,anchor=origin,hide
          axis,y=1.5cm,x=1cm]
          \addplot[color=red] table{data-FT-g-hann.tsv};
          \addplot[color=blue] table{data-FT-g-blackman.tsv};
        \end{axis}
      \end{tikzpicture}
    \end{minipage}
}
\caption{\emph{Left:} the Hann window $g_{\mathit{hann}}\in V^1_+(\R)$ (red) and the Blackman window $g_{\mathit{blac}}\in V^1_+(\R)$
  (blue). \emph{Right:} The Fourier transform of the Hann window
  $\widehat{g}_{\mathit{hann}}$ (red) and of
  the Blackman window $\widehat{g}_{\mathit{blac}}$ (blue). Both windows and
  their Fourier transforms are real and symmetric.}
\label{fig:ex1-windows-and-FT}
\end{figure}

 As an
example, let us consider the modulation parameter $b=3/5$. By
definition of $\kmax$ we get $\kmax=1$. Thus, the dual windows $h_z$
defined in (\ref{eq:def-h}) will have support in $\itvcc{-2}{2}$.

Since $g$ is a trigonometric polynomial on $\itvcc{-1}{1}$, it is
natural to take $z$ to be a trigonometric polynomial as well. For the
Hann window the standard choice~\eqref{eq:z-standard} is: 
\begin{equation}
z_{\mathit{hann}}(x)= b\cos(\pi x) \qquad \text{for } x \in \itvcc{0}{1},
\label{eq:hann-ex-z-def}
\end{equation}
while \eqref{eq:z-standard} for the Blackman window becomes:
\begin{equation}
z_{\mathit{blac}}(x)= b \bigl[-0.16+ 0.5 \cos(\pi x) + 0.08 \cos(2 \pi
  x) \bigr] \qquad \text{for } x \in \itvcc{0}{1}.
\label{eq:blackman-ex-z-def}
\end{equation}

Figure \ref{fig:g-ex1} shows dual windows of the Hann window
$h_{\mathit{hann}}$ and of the  Blackman  window $h_{\mathit{blac}}$ defined 
using $z$ from \eqref{eq:hann-ex-z-def} and
\eqref{eq:blackman-ex-z-def}, respectively. While $z_{\mathit{hann}}$ is
anti-symmetric around $x=1/2$, this is not the case for the chosen $z_{\mathit{blac}}$;
see Lemma~\ref{lem:symmetry} and Remark~\ref{rem:cor-on-smoothness}(c).
\begin{figure}
  \centering
  \begin{tikzpicture}[scale=2,y=1.5cm,x=1cm]
    \draw[->, thick] (-3.1,0) -- (3.1,0) node[right] {$x$}; \draw[->,
    thick] (0,-0.1) -- (0,1.2); \foreach \x in {-3,-2,-5/3,-1,0,1,5/3,2,3}
    \draw (\x,0) -- (\x,-0.1) node[below] {\x}; \foreach \y in {1}
    \draw (0,\y) -- (-0.1,\y) node[left] {\y}; 
    \begin{axis}[mark size=1.25pt,anchor=origin,hide
      axis,y=1.5cm,x=1cm]
\addplot[color=red] table{data-h-hann.tsv};
\addplot[color=blue] table{data-h-blackman.tsv};
    \end{axis}
  \end{tikzpicture}
\caption{Dual windows $h_{\mathit{hann}}$ (red) and $h_{\mathit{blac}}$ (blue) of the Hann and
Blackman window based on $z_{\mathit{hann}}$ and $z_{\mathit{blac}}$ defined in
\eqref{eq:hann-ex-z-def} and \eqref{eq:blackman-ex-z-def}, respectively. Both windows are in $C^1(\R)$ and with support $\supp h =
  \itvcc{-2}{-5/3}\cup \itvcc{-1}{1} \cup \itvcc{5/3}{2}$. }
\label{fig:g-ex1}
\end{figure}
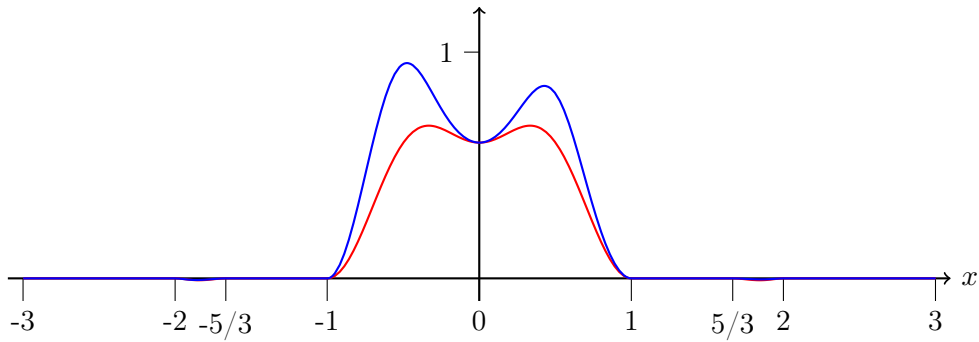

We can actually decrease the support size of the dual windows without
sacrificing the $C^1$-smoothness. By applying
Theorem~\ref{thm:small-support} with $N=1$ and taking $z_{\mathit{mid}}$ to be
the unique third degree trigonometric polynomial
$z_{\mathit{mid}}(x)=c_1 \cos(\pi x)+c_3 \cos(3 \pi x)$ so that $z \in C^1(\R)$, we
obtain dual windows in $C^1(\R)$ with support in $\itvcc{-1}{1}$. It
turns out that the support of the dual windows even shrink to $\itvcc{-2/3}{2/3}$ for
this specific setup. Since the two constructed functions $z_{\mathit{mid}}$ are
anti-symmetric around $x=1/2$, it follows by Theorem~\ref{thm:small-support}\eqref{item:small-symm}
that both these dual windows will be symmetric.  The short-support
dual windows of the Hann and Blackman window and their Fourier transform are shown in
Figure~\ref{fig:ex1-short-dual-windows-and-FT}.

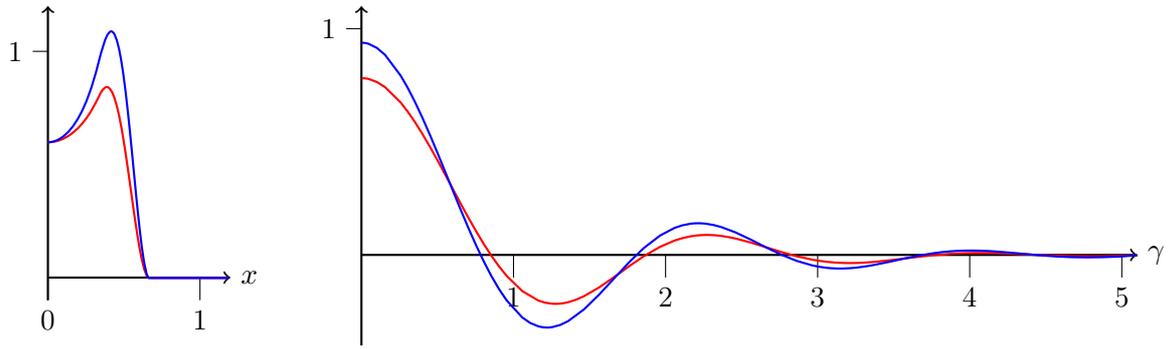
\begin{figure}[!htb]
{\centering
    \begin{minipage}[t][][b]{0.2\linewidth}
  \begin{tikzpicture}[scale=2,y=1.5cm,x=1cm]
    \draw[->, thick] (0,0) -- (1.2,0) node[right] {$x$}; \draw[->,
    thick] (0,-0.1) -- (0,1.2); \foreach \x in {0,1}
    \draw (\x,0) -- (\x,-0.1) node[below] {\x}; \foreach \y in {1}
    \draw (0,\y) -- (-0.1,\y) node[left] {\y}; 
    \begin{axis}[mark size=1.25pt,anchor=origin,hide
      axis,y=1.5cm,x=1cm]
\addplot[color=red] table{data-h-short-hann.tsv};
\addplot[color=blue] table{data-h-short-blackman.tsv};
    \end{axis}
  \end{tikzpicture}
    \end{minipage}
    \hspace{1em} 
    \begin{minipage}[t][][b]{0.75\linewidth}
      \begin{tikzpicture}[scale=2,y=1.5cm,x=1cm]
        \draw[->, thick] (0,0) -- (5.1,0) node[right] {$\gamma$};
        \draw[->, thick] (0,-0.4) -- (0,1.1); \foreach \x in {1,2,3,4,5}
        \draw (\x,0) -- (\x,-0.1) node[below] {\x}; \foreach \y in {1}
        \draw (0,\y) -- (-0.1,\y) node[left] {\y};
        \begin{axis}[mark size=1.25pt,anchor=origin,hide
          axis,y=1.5cm,x=1cm]
          \addplot[color=red] table{data-FT-h-short-hann.tsv};
          \addplot[color=blue] table{data-FT-h-short-blackman.tsv};
        \end{axis}
      \end{tikzpicture}
    \end{minipage}
}
\caption{\emph{Left:} Dual windows in $C^1(\R)$ with short support on $\itvcc{-2/3}{2/3}$. The dual
  of the Hann window is shown in red, while the dual of the Blackman
  window is shown in blue. \emph{Right:} The Fourier transform the two dual
  windows shown left: the Fourier transform of the dual Hann window
  (red) and the Fourier transform of
  the dual Blackman window (blue). Both dual windows and
  their Fourier transforms are real and symmetric.}
\label{fig:ex1-short-dual-windows-and-FT}
\end{figure}
\end{example}

The next example illustrates the construction of dual windows when $g$
does not have zero derivatives at the origin and the redundancy $(ab)^{-1}=3\pi/7$ is irrational. 
\begin{example}
\label{ex:example-oscill}
 We take $\beta$ to be a spline defined as:
\[ 
   \beta(x) =
   \begin{cases}
     p(x)  & x \in \itvcc{-1}{-4/5}, \\
     1    & x \in \itvcc{-4/5}{4/5}, \\
     p(-x) & x \in \itvcc{4/5}{1}, \\ 
     0 & \text{otherwise},
\end{cases}
\]
where $p(x)=10625- 60000x+ 135000x^2-  151250x^3 + 84375 x^4 -
18750x^5$ is the five-degree polynomial satisfying
$p(1)=p'(1)=p''(1)=p'(4/5)=p''(4/5)=0$ and $p(4/5)=1$. Then $\beta \in C^2(\R)$ is a bump function supported on
$\itvcc{-1}{1}$. We consider the window $g \in V^2_+(\R)$ defined by
\begin{equation*}
g(x)= \frac{1}{16} \bigl(2-(x-5)(x+3)\bigr) \beta(x).
\end{equation*}
As an example of an irrational modulation parameter, let us consider
$b=\frac{7}{3\pi}$. Then $\kmax=2$ so the support of $h_z$ is:
\begin{equation*}
  \supp{h_z} = \itvcc{-3}{-6\pi/7} \cup \itvcc{-2}{-3\pi/7} \cup
  \itvcc{-1}{1} \cup \itvcc{3\pi/7}{2} \cup \itvcc{6\pi/7}{3} \subset \itvcc{-3}{3}.
\end{equation*}

We chose $z$ to be the unique polynomial
of degree five that satisfies the six conditions of Theorem~\ref{thm:higher-order-smooth} for $n=2$; these conditions are
explicitly computed in Example~\ref{ex:comp-higher-order-cond}.
It follows that $h_z \in C^2(\R)$. The graphs of $g$ and the dual window
$h_z$ are shown in Figure~\ref{fig:g-ex2}.

  \begin{figure}
  \centering
  \begin{tikzpicture}[scale=2,y=1cm,x=1cm]
    \draw[->, thick] (-3.2,0) -- (3.2,0) node[right] {$x$}; \draw[->,
    thick] (0,-0.6) -- (0,1.4); \foreach \x in {-3,-2,-1,1,2,3}
    \draw (\x,0) -- (\x,-0.1) node[below] {\x}; \foreach \y in {-0.5,0.5,1}
    \draw (0,\y) -- (-0.1,\y) node[left] {\y}; 
    \begin{axis}[mark size=1.25pt,anchor=origin,hide
      axis,y=1cm,x=1cm]
\addplot[color=red] table{data-ex2-g-500.tsv};
\addplot[color=blue] table{data-ex2-h-500.tsv};
    \end{axis}
  \end{tikzpicture}
\caption{The window function $g \in V_+^2(\R)$ (red) and a dual window $h_z$ (blue)
  for $b=\frac{7}{3\pi}$. Both $g$ and $h$ are $C^2$-functions, and $h$
 has support in $[-3,3]$.}
\label{fig:g-ex2}
\end{figure}
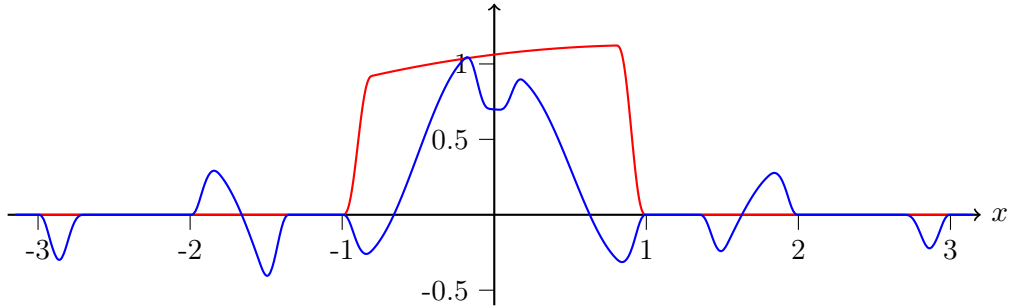
\end{example}


\end{document}